\documentclass[reqno]{amsart}
\usepackage[utf8]{inputenc}
\usepackage[margin = 1in]{geometry}
\usepackage{latexsym,amsmath,color,amsthm,epsfig,mathrsfs,mathdots}
\usepackage{mathtools}
\mathtoolsset{showonlyrefs}
\usepackage{graphicx}
\usepackage{amssymb}
\usepackage{mdframed}
\usepackage{fancyhdr}
\usepackage{tikz-cd}
\usepackage{verbatim}
\usepackage{nicefrac}
\usepackage{bbm}
\usepackage{tikz}
\usetikzlibrary{calc,decorations.pathmorphing}
\usepackage{hyperref}
\hypersetup{
    colorlinks = true,
    citecolor=black,
    filecolor=black,
    linkcolor=black,
    urlcolor=blue
}
\usepackage{amsfonts, esint}
\usepackage{color}
\usepackage{dsfont}
\usepackage[T1]{fontenc}
\usepackage[utf8]{inputenc}

\numberwithin{equation}{section}

\usepackage{lastpage}
\newtheorem{thm}{Theorem}[section]
\newtheorem{prop}[thm]{Proposition}
\newtheorem{cor}[thm]{Corollary}

\newtheorem{lem}[thm]{Lemma}

\newtheorem{preremark}[thm]{Remark}

\newenvironment{remark}{\begin{preremark}\rm}{\medskip \end{preremark}}
\numberwithin{equation}{section}

\newcommand{\R}{\mathbb R}
\newcommand{\eps}{\varepsilon}
\newcommand{\vphi}{\varphi}

\newcommand{\ds}{\displaystyle}

\DeclareMathOperator*{\osc}{osc}

\DeclareMathOperator{\tr}{tr}

\definecolor{sh}{RGB}{255,0,100}

\newcounter{case}
\renewcommand{\thecase}{\Alph{case}}
\newcounter{proofcase}[case]
\renewcommand{\theproofcase}{(\thecase\arabic{proofcase})}
\newif\ifusedcase

\newcommand{\proofcase}{%
  \ifusedcase\else\usedcasetrue\stepcounter{case}\fi
  \par
  \refstepcounter{proofcase}
  \everypar=\expandafter{\the\everypar{\setbox0=\lastbox}\everypar{}Case \theproofcase\ }%
}

\begin{document}

\begin{abstract}
We study regularity properties for solutions to the nakedly degenerate elliptic equation $a_{ij}\partial_{ij}u =0$, where the coefficients satisfy $I \ge a_{ij}(x) \ge \lambda(x) I$ and the only assumption is that $\lambda^{-1} \in L^p$. We prove an improvement of oscillation and a Liouville theorem for $p>d-1$, and a Harnack inequality for $p$ sufficiently large depending on dimension. Along the way, we obtain a new $\log-L^\eps$ Weak Harnack inequality for supersolutions. Then, touching subsolutions by double exponential blow-up barriers, we also derive a logarithmic local maximum principle that is new even in the uniformly elliptic case. Both of these results hold for $p>d-1$. Finally, we construct examples showing that there cannot be Harnack or Weak Harnack inequalities in the regime $p<d-1$, nor can there be power-type $L^\eps$ inequalities in the case of any $p<\infty$. 
\end{abstract}

\title{Harnack inequality for non-uniformly elliptic equations in non-divergence form}
\author{David Bowman}
\address[David Bowman]{Department of Mathematics, University of Chicago,  Chicago, Illinois 60637, USA}
\email{dbowman@uchicago.edu}
\maketitle
\section{Introduction}
\label{sec: intro}
In this paper, we set forth the study of Harnack inequalities for  degenerate elliptic equations of the form 
\begin{equation}
    \label{eq: degen}
    a_{ij}\partial_{ij} u =0 \text{ in } B_1 \subset \R^d,
\end{equation}
under the assumption that the coefficients satisfy
\begin{equation} 
\label{eq: ellipticityassumption}
\lambda(x) I \le a_{ij}(x)  \le I \text{ in } B_1
\end{equation}
where the ellipticity $\lambda>0$ is only assumed to satisfy
\begin{equation}
    \frac1\lambda \in L^p(B_1)
\end{equation}
for some $p>d-1$. Of course, the case $p=\infty$ corresponds to the classical uniformly elliptic case. We investigate which results from the classical regularity theory hold under the assumption that the ellipticity is allowed to deterioriate on small sets, controlled only in an $L^p$ sense. As it turns out, some results fail when $\lambda^{-1}$ is allowed to be unbounded. One surprising departure that we will see is that power-type Weak Harnack inequalities for supersolutions \textit{never hold} for any $p<\infty$. 

The equation \eqref{eq: degen} has the following probabilistic interpretation. Given a domain $D$ and a point $x$, consider a particle with dynamics 
$$
\begin{cases}
    dX_t = \sigma(X_t) \ dB_t\\
    X_0 = x
\end{cases}
$$
where $B_t$ refers to a standard Brownian motion, and the matrix field $\sigma$ satisfies $\sigma(y) \sigma^T(y) = 2A(y)$. Then, with $\tau_D$ denoting the exit time for the domain $D$, the solution $u$ to \eqref{eq: degen} has the representation
$$u(x) = \mathbb{E}^x[u(X_{\tau_D})].$$

Thus, a Harnack inequality would assert that, despite the highly anisotropic nature of the diffusion and the presence of pockets of low ellipticity in which particles may become stuck for long durations, the mixing effect is still strong enough to ensure that the exit probabilities starting from different points remain comparable. The difficulty will be in overcoming the ``traps" of low ellipticity that can effectively disconnect the values of solutions between regions. 

We will see that there exists a range of $p\in (p(d), \infty]$ for which a Harnack inequality does indeed hold, as well as that Harnack does \textit{not} hold for $p<d-1$. The range $d-1 \le p \le p(d)$ is left open. We conjecture that a Harnack inequality should be true for $p>d-1$, which we believe should be enough ellipticity to prevent the formation of thin regions with weak diffusion that effectively disconnect the values of solutions.

\subsection{Main Results}
Before stating our main results, we introduce the scale-invariant quantity 
$$\Gamma(B) \coloneqq \left(\frac{1}{|B|}\int_B \lambda^{-p}\right)^{1/p},$$
which, naturally, captures the average degeneracy of the ellipticity on a ball $B$. All of our estimates will depend on $\Gamma(B)$, which can be extremely large for balls of small radius. Also, to generalize our results as much as possible, we state them in terms of the Pucci extremal operators, which we define on the set $\mathbb{S}^d$ of symmetric $d\times d$ matrices as follows. For every $M \in \mathbb{S}^d$, the upper Pucci extremal operator is given by 
$$\mathcal{P}^+_{\lambda(\cdot),1}(M) \coloneqq \tr(M^+) - \lambda(\cdot)\tr(M^-),$$
while the lower extremal operator is given by 
$$\mathcal{P}^-_{\lambda(\cdot),1}(M) \coloneqq \lambda(\cdot) \tr(M^+) - \tr(M^-).$$
It is clear that any solution to a linear equation of the form $a_{ij}\partial_{ij} u=0$ with $\lambda(\cdot) I \le a_{ij}\le I$ satisfies 
$$\mathcal{P}^-_{\lambda(\cdot), 1}(D^2u)\le 0 \le \mathcal{P}^+_{\lambda(\cdot),1}(D^2u).$$
Further properties of the Pucci extremal operators can be found in \cite{CaffarelliCabre1995}. 

Our results hold for viscosity solutions. We briefly discuss this notion of solution later on in this introduction, but the interested reader should again refer to \cite{CaffarelliCabre1995} for a more comprehensive treatment. Because we prove results for the broadest class of viscosity solutions, we are also required to assume that $\lambda(\cdot)$ is continuous and positive in order to rule out pathological counterexamples. We state the two key assumptions on $\lambda(\cdot)$ for posterity. They are stated for general domains $U\subset \R^d$. The assumptions are:
\begin{equation}
    \label{eq: structuralassumption}
    \lambda(\cdot) \text{ is continuous and positive inside } U,
\end{equation}
and
\begin{equation}
    \label{eq: keyassumption}
    \lambda^{-1}(\cdot) \in L^p(U).
\end{equation}
It is important to stress that our results do not depend on any kind of lower bound on the ellipticity $\lambda(\cdot)$, nor on a modulus of continuity. Our estimates depend only on $\|\lambda^{-1}\|_{L^p(B_5)}$, and therefore apply to equations with regions of arbitrarily small ellipticity. The requirement that $\lambda(\cdot)$ be continuous and strictly positive may be dropped if one considers $L^s$-strong solutions. We will discuss different notions of solution in Section \ref{sec: othernotions}. 

Our first main result is the following $\log-L^\eps$ inequality, also known as a \textit{Weak Harnack inequality}. It is considerably weaker than the Weak Harnack inequality available in the uniformly elliptic case, or in the degenerate elliptic case considered in \cite{ImbertSilvestre2016}. 
\begin{prop}
\label{prop: logL^eps}
Let $u: B_5 \to \R$ be a nonnegative viscosity supersolution to 
$$\mathcal{P}^-_{\lambda(\cdot),1}(D^2u) \le 0 \text{ in } B_5.$$

Assume that the ellipticity $\lambda(\cdot)$ satisfies \eqref{eq: structuralassumption} and \eqref{eq: keyassumption} for some $p>d-1$. Then for any $$0 < \eps < \ds \frac{p-(d-1)}{d},$$
there holds
$$\int_{B_1} \left(\log \left(\frac{u}{\inf_{B_1}u}\right)\right)^\eps \le C(p,d, \eps) \Gamma(B_5)^{ \eps \gamma}.$$

The exponent $\gamma$ is given by $\ds \gamma \coloneqq \frac{pd}{p-(d-1)}$.
\end{prop}
The proof of this $\log-L^\eps$ estimate must contend with several difficulties that are not present in the uniformly elliptic case. Following Savin in \cite{Savin2005}, we touch a supersolution from below with paraboloids to deduce a measure-to-pointwise lemma that decays with a power of $\Gamma(B)$ (see Lemma \ref{lem: weakharnack1}). We then prove a doubling estimate in Lemma \ref{lem: doubling} by touching from below by singular cusps of the form $|x|^{-\beta}$, as in the paper \cite{ArmstrongSmart2014} by Smart and Armstrong. However, the exponent $\beta$ itself scales like a power of $\Gamma(B)$. This leads to a slightly unusual proof of the Weak Harnack inequality. Indeed, in applying an ink-spots argument, one covers $\{u>1\}$ by balls $B$ and seeks to show that $\{u \le M\} \cap B$ comprises a fixed amount of each ball $B$. In our case, however, we have that for each ball $B$ of the covering, the constant $M$ actually grows exponentially in $\Gamma(B)$, and the proportion of $B$ belonging to $\{u\le M(\Gamma(B))\}$ decays like a power of $\Gamma(B)^{-1}$. Thus balls of bad ellipticity are simultaneously the most expensive in level and the worst in measure gain, and the standard geometric iteration breaks down. 

In Section \ref{sec: optimality} we show directly that a uniform $L^\eps$-style estimate of the form $$\int_{B_1} u^\eps \lesssim (\inf_{B_1} u)^\eps$$
does not hold for any $\eps>0$, outside of the uniformly elliptic case. As a result, $\log$-moments are essentially the best one could hope to bound for supersolutions to $\mathcal{P}^-_{\lambda(\cdot),1}(D^2u)\le 0$, although it is possible that the range of admissible exponents for the $\log$-moments in Proposition \ref{prop: logL^eps} could be improved. Going even further, we show in Proposition \ref{prop: noHarnack} that in the case $p<d-1$, a Weak Harnack inequality of \textit{any} form is impossible. In light of the above proposition, $p=d-1$ is the critical threshold for Weak Harnack inequalities. 

As a consequence of the $\log-L^\eps$ estimate, we derive a unit scale improvement of oscillation result for solutions.
\begin{prop}
\label{prop: IOL}
    Let $u:B_5 \to \R$ satisfy in the viscosity sense the inequalities
    $$\mathcal{P}^-_{\lambda(\cdot),1}(D^2u)\le 0 \le \mathcal{P}^+_{\lambda(\cdot),1}(D^2u) \text{ on } B_5,$$
    where the ellipticity $\lambda(\cdot)$ satisfies \eqref{eq: structuralassumption} and \eqref{eq: keyassumption} for some $p>d-1$. Then there holds
    $$\osc_{B_1} u \le (1-\theta) \osc_{B_5}u,$$
    where $\theta \simeq \exp(-c(d,p) \Gamma(B_5)^{\gamma}) \in (0,1)$, and $\gamma$ is as in Proposition \ref{prop: logL^eps}.
\end{prop}
This should be contrasted with the result of \cite{ArmstrongSmart2014}, wherein it was shown that an improvement of oscillation holds for the equation
$$a_{ij}\partial_{ij} u = f \in L^\infty(B_1)$$
under the assumption that $\lambda^{-1} \in L^p(B_1)$ for $p\ge d$. We improve this result in the homogeneous case, relaxing the condition to $p>d-1$. We believe this threshold should be essentially optimal, in light of the fact that there is no Weak Harnack estimate in the case $p<d-1$. 
As an immediate consequence of Proposition \ref{prop: IOL}, we arrive at the familiar Liouville's theorem.
\begin{thm}
    \label{thm: Liouville}
    Suppose that $u: \R^d \to \R$ is bounded and satisfies in the viscosity sense the inequalities
    $$\mathcal{P}^-_{\lambda(\cdot),1}(D^2u) \le 0 \le \mathcal{P}^+_{\lambda(\cdot),1}(D^2u) \text{ in } \R^d.$$
    Assume the ellipticity $\lambda(\cdot)$ satisfies \eqref{eq: structuralassumption} as well as that for some $p>d-1$, 
    $$\limsup_{R \to \infty} \frac{1}{|B_R|} \int_{B_R} \lambda^{-p} <\infty.$$
     Then the solution $u$ must be constant. 
\end{thm}
Notice that the assumption on the ellipticity in the above theorem is the natural replacement of the assumption that $\lambda^{-1}\in L^p(\R^d)$, which does not make sense under the assumption that $\lambda\le 1$.
\\\par Our next main result is a logarithmic local maximum principle for subsolutions. It is obtained by touching a subsolution from above by translations of double exponential blow-up barriers of the form $U(x) \coloneqq c\exp\exp((1-|x|^2)^{-1}).$ To our knowledge, this is the first time an estimate has been obtained by touching with this kind of barrier. In the uniformly elliptic case, it allows one to bound the maximum of a subsolution in terms of its log-moments, for any moment of order larger than $d-1$. The proof makes no use of any classical Alexandroff-Bakelmann-Pucci maximum principle.
\begin{prop}
\label{prop: LMP2}
Let $u:B_3 \to \R$ be a nonnegative viscosity subsolution to 
$$\mathcal{P}^+_{\lambda(\cdot),1}(D^2u)  \ge 0  \text{ in } B_3.$$
Assume that the ellipticity $\lambda(\cdot)$ satisfies \eqref{eq: structuralassumption} and \eqref{eq: keyassumption} for some $p>d-1$. Fix an exponent $q>d-1$, and define the exponent
$$\kappa \coloneqq \frac{p}{p-(d-1)}.$$
Then there holds the following homogeneous form of a local maximum principle:
$$\Gamma(B_3)^{(1-d)\kappa} \lesssim  \int_{\{u>M/2\} \cap B_3} \left(\log \left(\frac{2u}{\sup_{B_1}u}\right)\right)^{q\kappa}.$$

Supposing further that $\sup_{B_1} u>2$ and fixing an exponent $\tau>\kappa q$, where $\kappa$ is as in Proposition \ref{prop: logL^eps}, one may estimate
$$\sup_{B_1} u \le 2 \exp\left(c\Gamma(B_3)^{\frac{\kappa(d-1)}{\tau - q\kappa}} \left(\int_{\{u>1\} \cap B_3} (\log u)^{\tau}\right)^{\frac{1}{\tau - q\kappa}}\right)$$
for some constant $c$ depending on $q,p,d,$ and $\tau$. 
\end{prop}
This estimate ensures that if a subsolution is not too small at a point, then even its logarithm must concentrate in some integral sense, which is a strengthening of the classical $L^p$-type local maximum principle for uniformly elliptic equations (see, for instance, Theorem 4.8 of \cite{CaffarelliCabre1995}). We believe this estimate is essentially sharp. At least, and as we discuss after the proof of Lemma \ref{lem: LMP1}, the estimate cannot be substantively improved by touching from above with different shapes. This is due to a hard restriction encoded into the geometry of convex barriers which are blowing up towards the boundary.

Finally, we arrive at our most important result, which is a Harnack inequality for $p$ sufficiently large, depending on dimension. 

\begin{thm}
\label{thm: Harnack}
Let $u:B_5 \to \R$ be nonnegative and satisfy in the viscosity sense the inequalities
    $$\mathcal{P}^-_{\lambda(\cdot),1}(D^2u)\le 0 \le \mathcal{P}^+_{\lambda(\cdot),1}(D^2u) \text{ on } B_5,$$ Assume that the ellipticity $\lambda(\cdot)$ satisfies \eqref{eq: structuralassumption} and \eqref{eq: keyassumption} for some $p>p(d)$, where 
$$p(d) \coloneqq \ds  \frac{d-1}{2}\left(d+2 + \sqrt{d(d+4)}\right).$$
Then there exists a constant $C = C(p,d,\Gamma(B_5)))$ such that 
$$\sup_{B_1}u \le  C \inf_{B_1}u.$$
The constant $C$ may be taken as  
$$C = \exp(c(d,p)\Gamma(B_5)^\nu)$$
for any 
$$\nu > \frac{dp^2}{(p-(d-1))^2- dp(d-1)}.$$
\end{thm}
We strongly suspect that this result is not optimal. In light of the improvement of oscillation (Proposition \ref{prop: IOL}), we believe that the correct threshold for Harnack should be $p>d-1$. Our Weak Harnack inequality (Proposition \ref{prop: logL^eps}) and local maximum principle (Proposition \ref{prop: LMP2}) both hold for $p>d-1$, but seem to only be compatible for $p$ sufficiently large. To further support our conjecture, we show in Section \ref{sec: optimality} that a Harnack inequality \textit{cannot} hold in the regime $p<d-1$, at least in dimension $3$ and higher. We also remark that, in contrast with the uniformly elliptic case, our improvement of oscillation and Harnack inequality results do not imply any kind of H{\"o}lder modulus of continuity at small scales. They do, however, provide large scale H{\"o}lder estimates, which could be useful for future work on the homogenization of the equation \eqref{eq: degen}. As we will discuss at the end of the paper, it is possible to remove the assumption that $\lambda(\cdot)$ be continuous and strictly positive if one considers a stronger notion of solution. The following table summarizes most of our new results.

\begin{table}[b]
\centering
\renewcommand{\arraystretch}{1.3}
\begin{tabular}{c|c|c|c}
 & $p<d-1$ & $d-1 < p \le p(d)$ & $p>p(d)$ \\
\hline
Harnack & fails & conjectured to hold & holds \\
\hline
Weak Harnack & fails & holds & holds
\end{tabular}
\caption{Status of Harnack and Weak Harnack estimates in the different integrability regimes.}
\label{tab:harnack-regimes}
\end{table}

\subsection{Historical Results}

The study of non-uniformly elliptic equations in divergence form has a long history. 
In the normalized setting $I \ge A \ge \lambda I$, Trudinger proved in \cite{Trudinger1971} 
that nonnegative weak solutions of
\[
\partial_i(a_{ij}\partial_j u)=0
\]
satisfy a Harnack inequality under the assumption $\lambda^{-1}\in L^p(B_1)$ for 
$p>\frac d2$. This was recently improved by Bella and Sch\"affner in \cite{BellaSchaeffner2021}, 
who obtained a Harnack inequality under the essentially optimal condition 
$p>\frac{d-1}{2}$. Their argument is based on a refinement of the classical Moser iteration technique. A construction by Franchi, Serapioni, and Serra Cassano in \cite{FranchiSerapioniCassano1998} shows that for $p<\frac{d-1}{2}$ Harnack inequality may fail, thus identifying $\frac{d-1}{2}$ as the critical exponent in divergence form. We also point out the related papers \cite{AlbrittonDong2023} and \cite{IgnatovaKukavicaRyzhik2014}, which obtain Harnack inequalities for uniformly elliptic equations with supercritical, divergence-free drifts via refined Moser iterations. Also relevant is the work \cite{FabesKenigSerapioni1982} by Fabes, Kenig, and Serapioni, who were able to derive an actual local H{\"o}lder regularity result for solutions to  divergence form equations with coefficients satisfying an $A_2$ Muckenhoupt weight condition. We refer the interested reader to \cite{BellaSchaeffner2021} for a more detailed discussion of the history of divergence form equations with degenerate ellipticity.

In contrast, the nondivergence case appears to be much less developed. Almost the entirety of the existing literature concerns non-uniformly elliptic equations possessing some nonlinear structural conditions, such as those which are of Monge-Ampere or $p$-Laplacian type. See for example \cite{CaffarelliGutierrez1997}, \cite{ImbertSilvestre2016}, and the references therein. The setting is very different here. Armstrong and Smart were the first to analyze plainly degenerate equations of the form \eqref{eq: degen} in the paper \cite{ArmstrongSmart2014}. In their very interesting paper, they establish unit-scale oscillation decay and stochastic homogenization results for fully nonlinear equations with bounded right-hand side, under the assumption that the ellipticity $\lambda$ satisfies \eqref{eq: keyassumption} for $p\ge d$. Their results do not, however, yield a Harnack inequality. 

More recently, Kim and Lee claimed in \cite{KimLee2025} a Harnack inequality for solutions to \eqref{eq: degen} under the assumption $\lambda^{-1}\in L^p(B_1)$ for some $p>2d$. Their proof unfortunately appears to be invalid. In particular, Theorem $1.3$ of their paper is an $L^\eps$ estimate for supersolutions, and we will show in our Proposition \ref{prop: noLeps} that actually no such estimate can hold for any $p<\infty$.

\subsection{Preliminaries}
In this subsection we give some brief remarks on the notion of solution, the assumptions on the ellipticity, and the notation used in the paper. \\\par
\textit{Remarks on viscosity solutions.}
For a general nonlinearity $F$, we say that $u$ is a viscosity supersolution to $F(D^2u)\le 0$ if it is lower-semicontinuous and if for every $\vphi$ which touches $u$ from below at some point $x_0$, there holds 
$$F(D^2\vphi(x_0))\le 0.$$
It is classical (see \cite{CaffarelliCabre1995}) that if $u$ is viscosity supersolution to 
$$F(D^2u) \le 0 \text{ on } B_1,$$
then its infimal convolution
$$u_\eps(x) \coloneqq \inf_{y \in B_1} u(y) + \frac{1}{2\eps}|x-y|^2$$
is a supersolution as well on the ball $B_{1-\delta}$ for $\delta= \delta(\eps)>0$ small, at least provided that $F$ is locally uniformly elliptic. There are two useful properties of this approximation that we will make use of in this paper. The first is that for each $\eps>0$, the function $u_\eps$ is uniformly semi-concave in $x$ and thus second-differentiable almost-everywhere by the Alexandroff theorem. The second useful property is that $u_\eps \le u$ and increases pointwise to $u$ as $\eps \to 0$. Viscosity subsolutions to $F(D^2u) \ge 0$ are defined analogously and are assumed to be upper-semicontinuous. This notion of viscosity solution is sometimes referred to by the term  ``$C-$viscosity solution". We will discuss different notions of solution in the final section of this paper. 
\\\par
\textit{Remarks on the assumptions.}
The assumption that $\lambda(\cdot)$ is continuous and positive rules out any pathological counter-examples to Harnack inequality or the existence theory. For instance, as pointed out in \cite{KimLee2025}, the function $u(x) = |x|$ is a viscosity solution to the degenerate elliptic equation
$$|x|^\alpha u''(x) =0 \text{ on } (-1,1)$$
for any $\alpha>0$, and $u$ certainly does not satisfy a Harnack inequality. As we discuss in Section \ref{sec: othernotions}, the infimal convolutions of $u(x)=|x|$ are \text{not} viscosity supersolutions, because all of the parabolas anchored near the origin touch precisely at the origin, where the equation provides no control. This is a striking new feature of the degenerate problem, and suggests that the $C$-viscosity notion of solution is incompatible with equations that allow for ellipticities vanishing at even just a single point. The discussion of which notions of solution allow for removing the strict positivity of the ellipticity is relegated to the end of the paper.

We do not assume or make use of any modulus of continuity for $\lambda(\cdot)$, which would naturally break down in any applications to homogenization. The estimates in our paper do not depend on any kind of pointwise lower bound on $\lambda(
\cdot)$. The natural way of phrasing our result then is that it is a uniform unit-scale Harnack estimate for a class of elliptic equations whose lower ellipticities may become arbitrarily small.

Lastly, we discuss the coefficient $\Gamma(B)$. It arises as follows. If $u$ is a solution to \eqref{eq: degen} (or a corresponding Pucci inequality) on $B_1$, then $u_r(\cdot):B_{1/r} \to \R$ defined by $u_r(x)\coloneqq u(rx)$ solves the corresponding equation with ellipticity $\lambda_r(\cdot)$ defined analogously. One sees easily that
$$\left(\frac{1}{|B_1|}\int_{B_1}\lambda_r^{-p}\right)^{1/p} = \left(\frac{1}{|B_r|}\int_{B_r} \lambda^{-p}\right)^{1/p}.$$
Hence, $\Gamma(\cdot)$ arises naturally when rescaling estimates. 
\\\par
\textit{Remarks on notation.}
Most of the notation we use is very standard. We wish to point out that at times we use the notation $A \lesssim B$ to denote $A \le cB$ for some absolute constant $c$. By ``absolute constant", we refer to constants depending on $p,d$, or nothing at all. Sometimes for clarity in stating estimates we refer to such constants by merely $c$, dropping their dependences which are of no interest to us. We clarify a few more pieces of notation for posterity. The expression $\ds \osc_A u$ denotes the oscillation of $u$ on a set $A$, or, in other words, $\ds \osc_A u \coloneqq \sup_A u - \inf_A u.$ We say that two symmetric matrices $M$ and $N$ have the ordering $M \ge N$ if all of the eigenvalues of $M -N$ are nonnegative. Lastly, $\lambda^{-1}$ always refers to the reciprocal $ \frac1\lambda$. 
\subsection{Acknowledgments}
The author would like to thank Luis Silvestre for many helpful discussions, as well as Sehyun Ji and Manos Katriadakis for reading an earlier version of this manuscript.
\section{On the optimality of some results}
\label{sec: optimality}
In this section we present examples showing that several standard estimates break down in the non-uniformly elliptic regime. By way of explicit construction, we show first that if the ellipticity $\lambda$ satisfies \eqref{eq: keyassumption} holds with $p<\infty$, then no $L^\eps$ inequality for supersolutions to \eqref{eq: degen} is valid for any $\eps>0$. This is consistent with the result of Armstrong, Silvestre, and Smart in \cite{ArmstrongSilvestreSmart2012}, who showed that in the uniformly elliptic setting $\lambda_0I \le a_{ij} \le I$, the $L^\eps$ estimate requires $\eps \lesssim \lambda_0$. Therefore in our setting, where the lower ellipticity $\lambda$ is allowed to become arbitrarily small, one should not expect any classical power-type $L^\eps$ estimate to hold.

Afterwards, we show that for every $p<d-1$, the Harnack inequality fails for the class of equations \eqref{eq: degen}. In fact, our construction shows that there is not even a \textit{weak} Harnack inequality in this regime.
\begin{prop}
\label{prop: noLeps}
Consider any dimension $d \ge 2$ and fix an exponent $p<\infty$. Then for every $\eta>0$ small, there exist  coefficients $a_{ij}^\eta$ and corresponding nonnegative viscosity supersolutions $u_\eta$ on the ball $B_{1/e}$ such that 
\begin{itemize}
\item $a_{ij}^\eta$ is continuous and uniformly elliptic in $x$ for each $\eta>0$,
\item $I \ge a_{ij}^\eta(x) \ge \lambda_\eta(x)  I$ where $\lambda_\eta$ satisfies 
$$\limsup_{\eta \to 0^+} \|\lambda_\eta^{-1}\|_{L^p(B_{1/e})}  <\infty,$$
\item $u_\eta$ is a supersolution to 
$$a_{ij}^\eta\partial_{ij}u_\eta \le 0 \text{ on } B_{1/e}$$
\item  $\ds \inf_{B_{1/(2e)}} u_\eta \simeq 1$ ,
\end{itemize}
and yet
$$\int_{B_{1/(2e)}} u_\eta^\eps \to +\infty \text{ as } \eta \to 0$$
for every $\eps>0$.

In fact, the supersolutions $u_\eta$ may be taken to be $C^{1,1}$ and hence $W^{2,\theta}$ strong-supersolutions.
\end{prop}
\begin{proof}
Fix $\beta>0$ satisfying 
\begin{equation}
\label{eq: firstbetaeq}
    0<\beta<\frac dp.
\end{equation}
We first search for a nonnegative supersolution $w(x) = w(|x|) \coloneqq \vphi(r)$ to the equation
$$a_{ij}\partial_{ij} w \le 0 \text{ on } B_{1/e} \setminus \{0\},$$
where the coefficient matrix is given by
$$A(x) \coloneqq r^\beta e_r\otimes e_r + (I-e_r\otimes e_r),$$
with $e_r$ denoting the radial direction. The supersolution inequality leads us to requiring that $\vphi$ satisfy 
$$r^\beta \vphi''(r) + \frac{d-1}{r} \vphi'(r) \le 0 \text{ for all } 0<r<\frac1e,$$
which leads us to try 
$$w(x) \coloneqq \vphi(r) \coloneqq e^{cr^{-\beta}}$$
for some $c>0$. Plugging in, $w$ is a supersolution on the punctured ball if
$$0 \ge a_{ij}\partial_{ij}w = r^{-2}\beta c \vphi\left[(\beta+1) + (c\beta -(d-1))r^{-\beta}\right].$$

This is certainly the case provided that $c\beta<(d-1)$ and
$$(1+\beta) + (c\beta -(d-1))e^{\beta}\le 0.$$

The inequality $1+\beta < e^\beta \le (d-1)e^\beta$ is true for any $\beta>0$ and $d\ge 2$, and now one takes $c$ even smaller if necessary to guarantee the above inequality. The constant $c$ depends on $\beta$, but for our purposes $\beta$ is a fixed exponent satisfying \eqref{eq: firstbetaeq}. We conclude that $w$ is a supersolution on the punctured ball, as desired.

Now, for $\eta>0$ small and fixed, we define the radius $r_\eta \coloneqq \eta^{1/\beta},$
and the coefficients

$$A^\eta(x)\coloneqq \begin{cases} A(x) & \text{ for } |x|>r_\eta, \\ \left(1- \frac{|x|}{r_\eta}\right) \eta I + \frac{|x|}{r_\eta} A(x) &\text{ for } |x|\le r_\eta.\end{cases} $$

The precise definition of $A^\eta$ on $B_{r_\eta}$ is not important for our purposes, only that it is uniformly elliptic and well-defined. The matrix $A$ then satisfies 
$$\lambda_\eta(x) I \le A(x) \le I \text{ in } B_{1/e}$$
where $\lambda(x) = \lambda(r)$ is given by 
$$\lambda_\eta(r) \coloneqq \begin{cases} r^\beta & \text{ for } r > r_\eta, \\  \eta & \text{ for } 0 \le r \le r_\eta.\end{cases}$$
One easily checks that 
$$\limsup_{\eta \to 0} \|\lambda_\eta^{-1}\|_p < \infty$$
precisely under the stated condition \eqref{eq: firstbetaeq} on $\beta$. 

Now we build our supersolution $u_\eta$ by truncation. For $M= M_\eta$ large and to be decided, we set
$$u_\eta  \coloneqq \min(w,M).$$

The truncation occurs on the sphere of radius $r =  c^{1/\beta}(\log M)^{-1/\beta}$. To ensure that $u_\eta$ is a viscosity supersolution with coefficients $A^\eta$, we choose $M$ through the relation $\eta \coloneqq c (\log M)^{-1}$. It is then immediate that on the whole ball $B_{1/e}$, there holds $a_{ij}^\eta \partial_{ij}u_\eta \le 0$ in the viscosity sense. It is also clear that
$$\inf_{B_{1/(2e)}} u_\eta  \simeq 1 \text{ for all  } \eta >0.$$
But for any $\eps > 0$, we may estimate the integrals as
 $$\left(\int_{B_{1/(2e)}} u_\eta^\eps\right)^{1/\eps} \simeq M (\log M)^{-d/(\beta \eps)} \to + \infty \text{ as } \eta \to 0,$$
 as claimed.

We now briefly discuss how a small modification allows us to make our supersolutions $C^{1,1}$. Instead of truncating $u_\eta \equiv M$ on $B_{r_\eta}$, we instead fit in a downward parabola cap. To be more precise, define $\gamma>0$ by $\gamma \coloneqq -\vphi'(r_\eta)$. Define $$\tilde{u}_\eta(x) \coloneqq \begin{cases} \vphi(r) & \text{ for }|x| \ge r_\eta \\ \psi(r) & \text{ for }|x|<r_\eta\end{cases},$$ where $\psi(r)\ds \coloneqq M + \frac{\gamma}{2} r_\eta - \frac{\gamma}{2r_\eta} r^2.$

One easily checks that $\tilde{u}_\eta$ is $C^{1,1}$. The concavity of the downward parabolic cap ensures that $\tilde{u}_\eta$ is a supersolution, and the same desired conclusions clearly hold.
\end{proof}
\begin{remark}
As we stated in the introduction, this construction shows that Theorem $1.3$ from \cite{KimLee2025} is in fact false. We believe the issue with the proof lies near the end, where the scaling of the ellipticity is not accounted for when applying Lemma $9.23$ from \cite{GilbargTrudinger2001}. In particular, their unit-scale estimate breaks down when rescaled to cubes of arbitrarily small length, leading to an incorrect ink-spots argument.
\end{remark}
We now turn to showing that a Harnack inequality cannot hold in general for equations of the form \eqref{eq: degen}, if the ellipticity $\lambda$ only satisfies \eqref{eq: keyassumption} for $p<d-1$. The construction will also show that Weak Harnack inequality fails.
\begin{prop} 
\label{prop: noHarnack} Consider any dimension $d \ge 3$ and fix an exponent $p<d-1$. Then for every $r>0$ small, there exist coefficients $a_{ij}^r$ and corresponding nonnegative viscosity solutions $u_r$ on the ball $B_1$ such that 
\begin{itemize}
\item $a_{ij}^r$ is uniformly elliptic \text{ in } $B_1$ 
\item $I \ge a_{ij}^r(x) \ge \lambda_r(x)  I$ where $\lambda_r$ is continuous and satisfies 
$$\ds \limsup_{r\to 0} \|\lambda_r^{-1}\|_{L^p(B_{1})}  <\infty,$$
\item $u_r$ is a $C^{1,1}$ viscosity solution to 
$$a_{ij}^r \partial_{ij}u_r = 0 \text{ on } B_{1}$$
\end{itemize}
and yet
$$\frac{\sup_{B_{1/2}}u_r}{\inf_{B_{1/2}}u_r} \to +\infty \text{ as } r\to 0,$$
showing that a Harnack inequality does not hold in the case $p<d-1$. In fact, the solutions $u_r$ can be chosen such that $\inf_{B_1} u_r \to 1$ as $r\to 0$, but 
$$| \{u_r \le N\}\cap B_1| \to 0 \text{ as } r \to 0$$
for any finite number $N$. Therefore, there is no possible \textit{Weak Harnack} inequality in the case $p<d-1$, either.
\end{prop}
\begin{proof}
We construct our example by smoothing out cusps in $d-1$ variables, which curve ever so slightly in the transverse direction. Fix $r>0$ small. With $p<d-1$ fixed, we first fix a parameter $\alpha \in (0,1)$ such that 
\begin{equation}\label{eq: alphaeq} (2-\alpha)p < d-1,
\end{equation}
and, with this choice of $\alpha$ being made, then fix a parameter $\sigma>0$ such that
\begin{equation}
    \label{eq: sigmaeq}
    0<\sigma < \frac{d-1 - (2-\alpha)p}{p}.
\end{equation}
We use the familiar notation $x\coloneqq (x', x_d)$, and set $\rho \coloneqq |x'|$. With this in hand, we define
$$u_{r}(x) \coloneqq 1-\delta x_d^2 + r^{-\sigma} \vphi_r(\rho),$$
where $\delta = \delta(r)>0$ is to be decided, and the function $\vphi_r(\cdot)$ is given by 
$$\vphi_r(\rho) \coloneqq \begin{cases}
    \frac\alpha2 r^{\alpha -2} \rho^2 & \text{ for }  0 \le \rho \le r,\\
    \rho^\alpha - \left(1-\frac\alpha2\right)r^\alpha & \text{ for } \rho \ge r.
\end{cases}$$
It is not hard to check that $u_r$ is $C^{1,1}$. Before we define our coefficients for which $u_r$ is a solution, we define a constant $\beta$ by fixing
\begin{equation}
\label{eq: betaeq}
    \alpha r^{-\sigma}((d-2) -\beta(1-\alpha)) = 2\delta
\end{equation}
We remark that, as $\delta = \delta(r) \to 0$, this amounts to choosing $\beta \to \frac{d-2}{1-\alpha}$, which is just a fixed number for our purposes. It is important to note that $\beta$ does not become arbitrarily small or large as $r\to 0$. 

We denote by $e_\rho$ the radial unit vector in the $x'$ variables. Now, we define the coefficients piecewise as follows:
$$A^r(x)\coloneqq \begin{cases} A_{\text{in}}(x) & \text{ for } |x'| \le r\\
A_{\text{out}}(x) & \text{ for } |x'|>r,
\end{cases}$$
where, with obvious abuses of notation,
$$A_{\text{out}}(x) \coloneqq \rho^{2-\alpha}\left(\beta e_\rho\otimes e_\rho + (I_{d-1} - e_\rho \otimes e_\rho)\right) + e_d \otimes e_d$$
and
$$A_{\text{in}}(x) \coloneqq 
\text{diag}(\lambda_{\text{in}}, \ldots, \lambda_{\text{in}}, 1),$$
with the constant $\lambda_{\text{in}}$ being given by 
$$\lambda_{\text{in}} \coloneq \ds r^{2-\alpha +\sigma}(\log 1/r)^{-1}.$$

Then, we certainly have
$$\lambda_r(x) I \le A_r(x)\le C I \text{ in } B_1,$$
where the constant $C$ depends only on $p,d,$ and $\alpha$, and
$$\lambda_r(x) \simeq \begin{cases} \lambda_{\text{in}}& \text{ for } |x'| \le r\\
(1-\frac{|x'|-r}{r})\lambda_{\text{in}} + \frac{|x'|-r}{r} |x'|^{2-\alpha} & \text{ for } r\le |x'| \le 2r,\\
|x'|^{2-\alpha} & \text{ for } |x'|>2r.
\end{cases}$$
The only point defining $\lambda_r(\cdot)$ on the annulus $r\le |x'| \le 2r$ in this way was to demonstrate that the minimal eigenvalue of the matrix $A$ can be minorized by a \textit{continuous} function satisfying the requisite $L^p$ bound. 

Integrating in cylindrical coordinates, one finds 
\begin{align} \int_{B_1} \lambda_r^{-p} &\lesssim_d \int_{|x'| \le 2r} r^{p(\alpha-2) - p\sigma}(\log(1/r))^p \ dx' + \int_{2r <|x'| <1} |x'|^{p(\alpha-2)} \ dx'\\
& \simeq_{d,p,\alpha} r^{p(\alpha-2) - p\sigma+ d-1}(\log 1/r)^{p} + 1
\end{align}
which stays bounded as $r\to 0$ precisely under the stated conditions \eqref{eq: alphaeq} and \eqref{eq: sigmaeq} on $\alpha$ and $\sigma$. 

Next, we check that $u_r$ is a viscosity solution on $B_1$, and in fact classical away from the interface $|x'|=r$. Indeed, for $|x'| <  r$, we simply have
$$a_{ij}^r \partial_{ij}u_r = \alpha  (d-1)\lambda_{\text{in}} r^{\alpha-2-\sigma} - 2 \delta=0$$
provided we choose
$$\delta = \delta(r) \coloneqq \frac{\alpha(d-1)}{2} (\log 1/r)^{-1}, $$
which certainly tends to $0$ as $r\to 0$. Next, we remark that that the map $x' \mapsto |x'|^\alpha$ has radial curvature equal to  $\alpha(\alpha -1) \rho^{\alpha-2}$ and $d-2$ tangential curvatures equal to $\alpha \rho^{\alpha-2}$. Then we may see that on $|x'|>r$, there holds
\begin{align} a_{ij}^r \partial_{ij}u_r &= \rho^{2-\alpha}r^{-\sigma}\left(\beta \alpha(\alpha-1) \rho^{\alpha -2} + \alpha(d-2) \rho^{\alpha -2}\right) - 2\delta\\
&= \alpha r^{-\sigma}(d-2 - \beta(1-\alpha)) - 2\delta =0
\end{align}
precisely under the stated condition \eqref{eq: betaeq}.

Finally, we briefly explain how to see that $u$ satisfies the definition of viscosity solution on the sphere $\{|x'|=r\}$. Fix such an $x_0$. We work in the cylindrical frame $e_\rho, \tau_1, \ldots \tau_{d-2}, e_d$, where $\tau_1, \ldots \tau_{d-2}$ are tangent to the sphere $\{|x'|=r\}$. One may check that $\partial_{\tau_i \tau_i} u_r$ and $\partial_{dd}u_r$ are continuous across the sphere, taking values
$$\partial_{\tau_i \tau_i} u_r = a \coloneqq \alpha r^{\alpha -2-\sigma}$$
$$\partial_{dd}u_r = -2\delta.$$
On the other hand, $\partial_{\rho \rho} u_r$ has a jump across the sphere. Indeed, 
$$\lim_{s \to 0^+} \partial_{\rho\rho}u(x_0 - se_\rho) = a$$
while
$$\lim_{s \to 0^+} \partial_{\rho\rho}u(x_0+ se_\rho) = b \coloneqq \alpha(\alpha-1)r^{\alpha-2-\sigma} <0<a.$$

Accordingly, if $\psi$ touches $u_r$ from above at $x_0 \in \{|x'|=r\}$, then, restricting $u_r - \psi$ to the inward normal half-line, the tangential lines, and the $x_d$-line, one sees that
$$\partial_{\rho\rho}\psi(x_0) \ge a, \qquad \partial_{\tau_i\tau_i} \psi(x_0) \ge a, \qquad \partial_{dd} \psi(x_0) \ge - 2\delta.$$
Thus
$$a_{ij}^r(x_0)\partial_{ij}\psi(x_0) = \lambda_{\text{in}}\left(\partial_{\rho \rho}\psi(x_0) + \sum_{i=1}^{d-2} \partial_{\tau_i\tau_i}\psi(x_0)\right) + \partial_{dd}\psi(x_0) \ge (d-1)\lambda_{\text{in}}a - 2\delta =0.$$

On the other hand, if $\psi$ touches $u_r$ from below at $x_0 \in \{|x'|=r\}$, then we restrict $u_r -\psi$ to the outward normal half-line and use otherwise analogous reasoning to find 
$$\partial_{\rho\rho}\psi(x_0) \le b, \qquad \partial_{\tau_i\tau_i} \psi(x_0) \le a, \qquad \partial_{dd} \psi(x_0) \le - 2\delta.$$
Then
$$a_{ij}^r(x_0) \partial_{ij}\psi(x_0) \le \lambda_{\text{in}}(b+ (d-2)a) -2\delta \le 0$$
since $b<a$. We conclude that $u_r$ is indeed a viscosity solution on the whole ball.

We conclude the proof by noticing that
$$\inf_{B_{1/2}} u_r = 1- \frac{\delta(r)}{4} \to 1 \text{ as } r\to 0,$$
while
$$\sup_{B_{1/2}}u_r = 1 + r^{-\sigma}\left(2^{-\alpha} -\left(1-\frac\alpha2\right)r^\alpha\right) \to +\infty \text{ as } r\to 0.$$
Furthermore, it is clear that $|\{u_r \le N\} \cap B_1|\to 0$ as $r\to 0$ for any finite $N<\infty$; the functions $u_r$ converge locally uniformly to $+\infty$ away from the axis $\{x'=0\}$. Indeed, one may compute that
$\{u_r \le N\} \cap B_1 \subset \{(x',x_d)\subset B_1 \ : \ |x'| \le C_N r^{\sigma/\alpha}\}$, assuming we also choose $\sigma<\alpha$. 
\end{proof} 
\begin{remark}
    There is no integral estimate for supersolutions in the regime $p<d-1$. Indeed, take any nonnegative, increasing, and unbounded function $f: \R \to \R$. Were it to be the case that for all nonnegative supersolutions to \eqref{eq: degen} satisfying $\ds \inf_{B_{1/2}} u =1$, there existed a constant $C = C(d,p, \Gamma(B_1))$ such that 
    $$\int_{B_{1/2}} f(u) \le C,$$
    we would derive for any such supersolution that $\ds | \{u \ge N\}\cap B_{1/2}| \le \frac{C}{f^{-1}(N)} \to 0$ as $N \to \infty$. The above construction contradicts this conclusion. 
\end{remark}
\section{The $\log-L^\eps$ estimate}
\label{sec: logLeps}
In this section, we establish a Weak Harnack type inequality for supersolutions. The first step is to produce a measure-to-pointwise estimate that says that if a supersolution is above a certain threshold on a large enough portion of a ball, then it cannot dip too low towards the center. It is a very weak improvement of minimum estimate, and its efficacy decreases drastically on balls of bad ellipticity, as measured by the coefficient $\Gamma(B)$. We follow the approach of Savin in \cite{Savin2005} and touch supersolutions from below by paraboloids. One departure from the uniformly elliptic setting is that the Jacobian of the contact-to-vertex map is no longer bounded, leading to an integrability requirement on $\lambda^{-1}$ that depends on the dimension. An estimate of this form cannot hold for $p<d-1$, as shown in Proposition \ref{prop: noHarnack}.
    \begin{lem}
        \label{lem: weakharnack1} Assume $u$ is a nonnegative supersolution to
        $$\mathcal{P}^-_{\lambda(\cdot),1}(D^2u) \le 0\text{ in } B_1$$
        where the ellipticity $\lambda(\cdot)$ satisfies \eqref{eq: structuralassumption} and \eqref{eq: keyassumption} for some $p>d-1$. Then there exists an absolute constant $c>0$ such that if $u(0)\le 1$, then 
        $$|\{u \le 2\} \cap B_1| \ge c \Gamma(B_1)^{-\frac{p(d-1)}{p-(d-1)}} |B_1|.$$
    \end{lem}
\begin{proof}
By replacing $u$ by its infimal convolution $u_\eps$, it is enough to consider $u$ semi-concave, and therefore first and second-differentiable except on a set $E\subset B_1$ of measure zero. Consider $U \coloneqq \{u>2\} \cap B_{1/4}$. For each $x \in U$ we look at the inf-convolution 
$$h(x) \ds \coloneqq  \inf_{y \in B_1} u(y) + 3|x-y|^2.$$

Now, since $u \ge 0$, one obtains for $x \in U$ and $y \in \partial B_1$ that $u(y) + 3|x-y|^2 \ge \frac{27}{16}.$ On the other hand, since $u(0)\le 1$, for any $x \in U$ there holds $h(x) \le 1+ 3|x|^2 \le \frac{19}{16}.$

Therefore, the minimizers $y$ must occur inside the ball $B_1$, away from the boundary. Furthermore, minimizers belong to the set $\{u<2\} \cap B_1$. And while minimizers may not be unique for each $x \in U$, the following criticality conditions are satisfied for any $x \in U$ with corresponding contact point $y$ not belonging to the null set $E$:

\begin{equation}
\label{eq: FOC1}
\nabla u(y) = 6(x-y),
\end{equation}
\begin{equation}
\label{eq: FOC2}
D^2u(y) \ge -6I
\end{equation}

Note that \eqref{eq: FOC1} completely determines the center $x$ in terms of the contact point $y$. 
Let $m$ be the map $m(y)=x$ determined by \eqref{eq: FOC1}, and let $\mathcal{T}$ be the set of values $y$ takes over $x \in U = m(\mathcal{T})$. It is standard to see that $\nabla u$ and $m$ are Lipschitz on $\mathcal{T}$, and we have the identity $Dm(y) = I + \frac16 D^2u(y)$. The goal is to estimate $\det(Dm(y)) = \det(I + \frac{1}{6}D^2u(y))$. Note that $Dm(y)$ is a nonnegative, symmetric matrix. 

To estimate $\det(Dm)$, we consider its eigenvalues. We first see that $D^2u(y)$ must have one non-positive eigenvalue, lest the supersolution inequality be violated. Indeed, if $D^2u$ were a strictly positive matrix, then $(D^2u(y))^-$ would vanish, and we would reach the nonsensical inequality
$$\lambda(y) \tr(D^2u(y)) \le 0.$$
In light of this fact, we see that according to \eqref{eq: FOC2}, the matrix $D^2u(y)$ has at least one eigenvalue between $-6$ and $0$, leading to the conclusion that $Dm(y)$ has at least one eigenvalue between $0$ and $1$. To estimate the remaining $d-1$ eigenvalues of $Dm(y)$, we again use the supersolution inequality to deduce that on $\mathcal{T}\setminus E$ there holds
$$\tr((D^2u(y))^+) \le \frac{\tr((D^2u(y))^-)}{\lambda(y)} \le \frac{6(d-1)}{\lambda(y)}.$$ 

This implies that the matrix $Dm(y)$ may be bounded by $|Dm(y)| \lesssim \lambda(y)^{-1}$. Combining these observations, we obtain that $\det(Dm(y)) \lesssim \lambda(y)^{1-d}$, for every $y\in \mathcal{T}\setminus E$. Now, define $\mu$ by the relation $|\{u \le 2\} \cap B_1| \coloneqq \mu $.  We will deduce a lower bound on $\mu$. Applying the change-of-variables formula, we obtain via H{\"o}lder's inequality the estimate

\begin{align}
(1-4^d\mu)|B_{1/4}| &\le |\{u>2\} \cap B_{1/4}|  = |U| = |m(\mathcal{T})| \le \int_{\mathcal{T}} |\det Dm(y)| \ dy\\
& \lesssim |\mathcal{T}|^{1-(d-1)/p} \Gamma(B_1)^{d-1} \lesssim \mu^{1-(d-1)/p} \Gamma(B_1)^{d-1},
\end{align} 
and the conclusion follows after solving for a lower bound on $\mu$. 

For a general lower-semicontinuous viscosity supersolution $u$, the conclusion follows by a simple $\Gamma$-convergence argument. See, for example, \cite{ImbertSilvestre2016}.
\end{proof}

We turn now to proving a doubling estimate. We adopt the idea of Smart and Armstrong in \cite{ArmstrongSmart2014} and touch from below by singular cusps of the form $|x|^{-\beta}$. The exponent $\beta$ will be connected to the average degeneracy coefficient $\Gamma(B_4)$. See also the paper \cite{Mooney2015}, in which a similar doubling lemma is proven by sliding this family of barriers.
\begin{lem}
\label{lem: doubling}
        Let $u$ be a nonnegative supersolution to
        $$\mathcal{P}^-_{\lambda(\cdot),1}(D^2u) \le 0 \text{ on } B_4,$$
        where the ellipticity $\lambda(\cdot)$ satisfies \eqref{eq: structuralassumption} and \eqref{eq: keyassumption} for some $p>d-1$. Then there exists a constant $M$, depending only on $p, d$, and $\Gamma(B)$, such that if 
        $$u \ge M \text{ on } B_1,$$
        then $u>1$ on $B_2$. The constant $M$ may be taken as $M = \exp(c \Gamma(B_4)^\kappa)$, where $\kappa \coloneqq \ds \frac{p}{p-(d-1)}$ and the constant $c$ is absolute.
\end{lem}
\begin{proof}
As in \cite{ArmstrongSmart2014}, the idea is that if at some point in $B_2\setminus B_1$, $u$ becomes too small relative to its lower bound on $B_1$, then $u$ may be touched from below by too many singular cusps with extremely large curvature ratio, which necessitates contact in the region where $\lambda$ is very small. The $L^p$ integrability assumption on $\lambda^{-1}$ is what will allow us to control the size of the contact set and reach a contradiction.

As is customary, we prove the lemma first under the assumption that $u$ is semi-concave. At the end, we explain how the inf-convolution approximation allows us to conclude the lemma for general supersolutions. As before, we remark that off of a set $E$ of measure zero, the function $u$ is first and second-differentiable.

Proceeding with the proof, we assume for the sake of contradiction that $u(x_0) \le 1$ for some $x_0 \in B_2$. Now, define the truncated cusps
$$\psi(x) \coloneqq C_1\begin{cases} 1 & \text{ on } B_{1/2} \\ \frac{|x|^{-\beta} - 3^{-\beta}}{1-3^{-\beta}} & \text{ on } \R^d \setminus B_{1/2}.
    \end{cases}
$$
where both $\beta>0$ is to be decided and $C_1 = C_1(\beta) \coloneqq \exp(c\beta)>0$ is chosen such that $\psi>2$ on $B_{5/2}$. The constant $c$ is absolute. We remark that $\psi<0$ on $\R^d \setminus \overline{B_3}$. 

Now fix an arbitrary center $x \in B_{1/2}$. Recenter and slide $\psi(x-y)$ from below until the graph touches the graph of $u$ at a contact point $y$. We argue that $y \in B_4 \setminus \overline{B_{1/2}(x)}.$  To that end, observe that if $w \in \partial B_4$, then the nonnegativity of $u$ and the above considerations imply that $u(w) - \psi(x-w)>0$. On the other hand, since $|x-x_0|<\ds \frac52$, there holds
$$u(x_0) - \psi(x_0-x) < 1-2<-1.$$

Therefore the contact point is not on $\partial B_4$. And since $u\ge M$ on $\overline{B_1} \supset  \overline{B_{1/2}(x)}$, the conclusion will follow by taking $M \coloneqq 2C_1(\beta)$. Note that we have still not yet chosen $\beta$. 

As in Lemma \ref{lem: weakharnack1}, we have the following first and second order conditions for each center $x\in B_{1/2}$ with corresponding contact point $y\notin E$:
\begin{equation}
\label{eq: FOC3}
\nabla u(y) = \nabla \psi(y-x),
\end{equation}
\begin{equation}
\label{eq: FOC4}
D^2u(y) \ge D^2 \psi(y-x).
\end{equation}
We now record that the Hessian of $|x|^{-\beta}$ diagonalizes as 
$$D^2|x|^{-\beta} \sim \beta |x|^{-2-\beta}\begin{cases} 1+ \beta & \text{ with multiplicity } 1,\\ -1 & \text{ with multiplicity } d-1.\end{cases}$$
At any contact point $y$, the supersolution inequality implies 

\begin{equation}
\label{eq: doublingbetaeq} \beta C_1 |x-y|^{-2-\beta} [(1+\beta) \lambda(y) - (d-1)] =\mathcal{P}^-_{\lambda(y),1}(D^2\psi(y-x)) \le 0.
\end{equation}
Note that \eqref{eq: FOC3} implies that the value of the contact point $y$ completely determines the center $x$. As a consequence, we can define (as before) a map $m$ from $y$ to $x$. As in the previous lemma, we define the contact set $\mathcal{T}$ be the values $y$ takes over all $x \in B_{1/2}$. Then, defining 
$$\mu  \coloneqq \ds \frac{d-1}{\beta +1},$$
we find that 
\begin{equation}
\label{eq: contactsmallellipticity} \mathcal{T} \subset \{y \in B_4 \ : \ \lambda(y) \le \mu\}.
\end{equation}
As in \cite{ArmstrongSmart2014}, the semi-concavity of $u$ and boundedness of $D^2\psi$ on $\mathcal{T}$ imply that $u$ is in fact $C^{1,1}$ on $\mathcal{T}$. Here we have used that contact points occur at a distance of at least $1/2$ from the vertices. This justifies the following differentiations.

Proceeding similarly to Lemma \ref{lem: weakharnack1}, we differentiate \eqref{eq: FOC3} with respect to $y$ to find 
$$D^2u(y) = D^2\psi(y-m(y))(I- Dm(y)) \text{ for all } y \in \mathcal{T} \setminus E$$
and consequently, now suppressing the arguments with the understanding that $y \in \mathcal{T}\setminus E$,
\begin{equation}
\label{eq: meq}
Dm = I - (D^2\psi)^{-1} D^2u
\end{equation}

We claim that $Dm$ has at least one eigenvalue between $0$ and $1$, although now we have to take extra care since $\psi$ is not concave. The claim will follow if we show that $(D^2\psi)^{-1}D^2u$ has an eigenvalue between $0$ and $1$. To that end, for each $t \in [0,1]$ let $\mu(t)$ denote the minimum eigenvalue of the symmetric matrix $D^2u - t D^2 \psi$. Arguing as in Lemma \ref{lem: weakharnack1}, we have that $\mu(0) \le 0$ because $u$ is a supersolution, and we have as well  that $\mu(1) \ge 0$ thanks to \eqref{eq: FOC4}. The map $\mu$ is clearly continuous, so we may conclude that $(D^2\psi)^{-1}D^2u$ has an eigenvalue between $0$ and $1$. Therefore, $Dm$ does as well. 

Furthermore, the contact inequality \eqref{eq: FOC4} may be combined with the assumption $\mathcal{P}^-_{\lambda(\cdot),1}(D^2u)\le 0$ to derive
$$\tr(D^2u(y)^+) \le \lambda(y)^{-1} (d-1)\beta C_1 |y-m(y)|^{-2-\beta} \text{ for all } y \in \mathcal{T}\setminus E$$ 
and, since $\lambda <1$, we arrive at
$$|D^2u(y)| \le 2\lambda(y)^{-1} (d-1)\beta C_1 |y-m(y)|^{-2-\beta} \text{ for all } y \in \mathcal{T}\setminus E.$$
On the other hand, for all $y$ there holds
$$|(D^2\psi)^{-1}(y-m(y))| \le (\beta C_1)^{-1} |y-m(y)|^{2+\beta}$$
from which we derive that for every $y \in \mathcal{T}\setminus E$ the estimate
\begin{equation}
    \label{eq: meq2}
    |Dm(y)| \lesssim \lambda(y)^{-1}.
\end{equation}
The matrix $Dm(y)$ may not be symmetric or diagonalizable, but the fact that it has at least one eigenvalue between $0$ and $1$ implies that it has at least one singular value between $0$ and $1$. The remaining $d-1$ singular values may bounded above using \eqref{eq: meq2}, allowing us to deduce for all $y \in \mathcal{T}\setminus E$ the determinant bound
$$|\det(Dm(y))| \lesssim \lambda(y)^{1-d}.$$

Now, since $\nabla u$ and $(\nabla \psi)^{-1}$ are Lipschitz on $\mathcal{T}$, so is $m$, and therefore it is legitimate to apply the change of variables formula to obtain
\begin{align}
|B_{1/2}| &= |m(\mathcal{T})| \le \int_\mathcal{T} |\det(Dm(y))| \ dy \lesssim \int_{\{\lambda <\mu\} \cap B_4} \lambda^{1-d}\\
&\lesssim | \{\lambda <\mu\} \cap B_4|^{1-(d-1)/p} \|\lambda^{-1}\|_{L^p(B_4)}^{d-1} \lesssim \mu^{p-(d-1)}\Gamma(B_4)^{p}.
\end{align}
Here we have used H{\"o}lder's and Chebyshev's inequalities, as well as the inclusion \eqref{eq: contactsmallellipticity}. We reach a contradiction if we let $\ds \mu =\eta \Gamma(B_4)^{-\frac{p}{p-(d-1)}}$ for some small constant $\eta$ depending only on dimension. Recalling that $\mu \simeq \beta^{-1}$ and $M = 2C_1 \simeq \exp(c\beta)$ we conclude the proof in the case that $u$ is a general semi-concave supersolution. 

For a general lower-semicontinuous supersolution $u$ satisfying the assumptions of the lemma, we remark that the inf-convolutions 
$$u_\eps(x) \coloneqq \inf_{y  \in B_4} u(y) + \frac{1}{2\eps}|x-y|^2 \ge M \text{ on } B_{1-\eta}$$
for $\eta \coloneqq \sqrt{2M\eps}.$

To see this, note that if $x \in B_{1-\eta}$ and $y \in B_4 \setminus B_1$, then 
$$u_\eps(x) \ge u(y) + \frac{1}{2\eps}|x-y|^2 \ge \frac{\eta^2}{2\eps}\ge M.$$

Applying the lemma in the already-proven case of semi-concave supersolutions, we conclude that $u_\eps >1$ on $B_{2-2\eta}$, up to multiplying $M$ by an absolute constant. Sending $\eps \to 0$ we deduce that $u>1$ on $B_2$ as claimed. 
\end{proof}

Combining a scaled version of the previous two estimates, we arrive at the following same-ball estimate. 
\begin{cor}
\label{cor: sameball}
 Let $u$ be a nonnegative supersolution to
        $$\mathcal{P}^-_{\lambda(\cdot),1}(D^2u) \le 0 \text{ on } B_1,$$
        where the ellipticity $\lambda(\cdot)$ satisfies \eqref{eq: structuralassumption} and \eqref{eq: keyassumption} for some $p>d-1$. Let $B = B_r(x_0)$ be any ball such that $2B \coloneqq B_{2r}(x_0) \subset B_1$, and suppose that $\ds \inf_B u \le 1$. Then there holds
$$|\{u \le \exp(c_1\Gamma(2B)^\kappa) \} \cap B| \ge c_2\Gamma(2B)^{(1-d)/(1-\alpha)} |B|$$
where we have defined $\alpha \ds \coloneqq   \frac{d-1}{p} \in (0,1)$. The constants $c_1$ and $c_2$ are absolute.
\end{cor}

At long last, we may derive a $\log-L^\eps$ estimate for supersolutions. As discussed in the introduction, the proof is considerably more difficult than in the uniformly elliptic case. This is because the above corollary is extremely weak. The problem posed by non-uniform ellipticity is two-fold: if $u$ is small at some point $x \in B$, then $u$ inherits an upper bound which grows exponentially with $\Gamma(B)$ on a portion of $B$ that decays polynomially with respect to the reciprocal of $\Gamma(B)$. This dual weakness makes any kind of ink-spots argument very delicate. At times we take extra care to keep track of constants depending on $\Gamma(B_5)$ for the purposes of large-scale regularity results. 

\begin{proof}[Proof of Proposition \ref{prop: logL^eps}]
The estimate in the statement of the proposition is homogeneous, so we may assume without loss of generality that $\inf_{B_1}u=1$. Take $t_0=1$, and then for $k \ge 1$ consider an increasing, unbounded sequence $\{t_k\}$ to be decided. For each $t_k$, we define the super-level set $A_k$ 
$$
A_k \coloneqq \{u>t_k\} \cap B_1,$$
which is open due to the lower-semicontinuity of $u$. For each $x \in A_k$, let $r = r(x)>0$ be the first radius such that $\partial B_{r}(x) \cap \{u\le t_k\} \cap B_1 \neq \emptyset$. Since $\inf_{B_1}u=1$ and $t_k \ge 1$, it is guaranteed that $r(x) \le 2$. Now we define the balls $\widetilde{B}_x \coloneqq B_{2r(x)}(x)$. Necessarily, $\widetilde{B}_x \subset B_5$ for each $x \in A_k$, and the balls $\{\widetilde{B}_x\}_{x \in A_k}$ clearly form a covering of $A_k$. Therefore, we may apply the Besicovitch covering theorem to reduce our covering of $A_k$ to $n(d)$ subcollections $\mathcal{C}_i$, $1\le i \le n(d)$, where each $\mathcal{C}_i$ is comprised of countably many disjoint balls $\widetilde{B}$. We refer to the balls $\widetilde{B}$ as \textit{outer balls}, and drop the subscript denoting the center $x$ wherever convenient. 

For each outer ball $\widetilde B$ in the Besicovitch subcovering, we sort $\widetilde B \in \mathcal{G}_k$ provided
$$
\Gamma(\widetilde B) \le N_k,
$$
for $N_k$ to be chosen. The sub-collection $\mathcal{G}_k$ refers to the \textit{good} outer balls, on which the ellipticity is controlled by a parameter of our choosing. We call the remaining balls \textit{bad}, and sort them into the sub-collection $\mathcal{B}_k$. We can control the size of the bad set by noting that for each of the $n(d)$ subcovers $\mathcal{C}_i$, we have
$$
N_k^p \left|\bigcup_{\widetilde B \in \mathcal{C}_i \cap \mathcal{B}_k}\widetilde B\right|
=
N_k^p \sum_{\widetilde B \in \mathcal{C}_i \cap \mathcal{B}_k} |\widetilde B|
\le
\sum_{\widetilde B \in \mathcal{C}_i \cap \mathcal{B}_k} \int_{\widetilde B} \lambda^{-p}
\le
\|\lambda^{-1}\|_{L^p(B_5)}^p,
$$
just by summing over disjoint outer balls.

To avoid technicalities in what follows, we define
$$
K_0 \coloneqq \sup\{k>0 \ : \ A_k \neq \emptyset\}.
$$
The following analysis holds for all $0 \le k \le K_0-1 \le +\infty$, and from now on it is assumed that $k$ falls in this range (which could be infinite). With this in mind, and using that there are at most $n(d)$ subcollections $\mathcal{C}_i$, we define
$$
N_k = c(d)\,\Gamma(B_5)\,|A_k|^{-1/p},
$$
where $c(d)$ is chosen to ensure that at least half of the set $A_k$ is covered by \textit{good} outer balls. In other words, with this choice of $N_k$ we ensure that
$$
\left|\bigcup_{i=1}^{n(d)} \bigcup_{\widetilde B \in \mathcal{C}_i \cap \mathcal{B}_k}\widetilde B\right|
<
\frac12|A_k|.
$$
As a consequence, we may apply the pigeonhole principle to select a subcover $\mathcal{C}_{i_*}$ such that
\begin{equation}
\label{eq: subcoverselection}
\left|\bigcup_{\widetilde B \in \mathcal{C}_{i_*} \cap \mathcal{G}_k}\widetilde B\right|
\ge
\frac{1}{2n(d)}|A_k|.
\end{equation}

Now, we need to choose carefully particular sub-balls of our covering. We claim that for each ball $\widetilde{B}$ of our covering, there exists a sub-ball $B \subset \widetilde{B} \cap B_1$ such that $|B| =4^{-d} |\widetilde{B}|$ and $\partial B \cap \{u \le t_k\} \cap B_1 \neq \emptyset$. To see this, fix an $x\in A_k$ and consider the ball $B_{2r(x)}(x) = \widetilde{B}_x = \widetilde{B}$, and let $z_x \in \partial B_{r(x)}(x) \cap \{u\le t_k\} \cap B_1$. The existence of $z_x$ is guaranteed by the definition of $r(x)$. Finally, let $y_x$ denote the midpoint between $x$  and $z_x$. The claim follows by taking $B \coloneqq B_{r(x)/2}(y_x)$, and from now on the notation $B$ always refers to such a ball associated to each ball $\widetilde{B}$ in the covering. See Figure \ref{fig:covering} for an illustration of this construction. Then, since the outer balls $\widetilde{B}$ in $\mathcal{C}_{i_*}$ are disjoint, their associated \textit{inner balls} $B$ are also disjoint, and we have
\begin{equation}
\label{eq: innerselection}
\left|\bigcup_{\widetilde B \in \mathcal{C}_{i_*} \cap \mathcal{G}_k}B\right|
=
4^{-d}
\left|\bigcup_{\widetilde B \in \mathcal{C}_{i_*} \cap \mathcal{G}_k}\widetilde B\right|
\ge
\frac{1}{2^{2d+1}n(d)}|A_k|.
\end{equation}

\begin{figure}[t]
    \centering
    \begin{tikzpicture}[scale=2.2]

\coordinate (O) at (0,0);


\def\blobpath{
    (-0.95, 0.10)
    .. controls (-0.92, 0.55) and (-0.80, 0.78) ..
    (-0.50, 0.85)
    .. controls (-0.20, 0.90) and (0.05, 0.78) ..
    (0.18, 0.58)
    .. controls (0.28, 0.42) and (0.42, 0.30) ..
    (0.55, 0.22)
    .. controls (0.72, 0.12) and (0.82, -0.05) ..
    (0.78, -0.25)
    .. controls (0.72, -0.48) and (0.50, -0.60) ..
    (0.25, -0.65)
    .. controls (0.00, -0.72) and (-0.35, -0.78) ..
    (-0.60, -0.82)
    .. controls (-0.82, -0.85) and (-0.96, -0.65) ..
    (-0.98, -0.40)
    .. controls (-1.00, -0.15) and (-0.97, 0.00) ..
    (-0.95, 0.10)
}

\begin{scope}
    \clip (O) circle (1);
    \fill[red!18] \blobpath -- cycle;
\end{scope}

\draw[thick] (O) circle (1);

\begin{scope}
    \clip (O) circle (0.98);  

    \draw[red!40!black, line width=0.4pt]
        \blobpath -- cycle;

    \draw[thick, red!40!black,
          decoration={random steps, segment length=3pt, amplitude=0.8pt},
          decorate]
        \blobpath -- cycle;
\end{scope}


\node[font=\large] at (-0.52, 0.45) {$A_k$};

\node[font=\large] at (0.45, 0.65) {$B_1 \!\setminus\! A_k$};

\node[font=\normalsize, anchor=north west] at ($(O)+(320:1.05)$) {$B_1$};


\coordinate (x) at (-0.20, -0.02);

\coordinate (zx) at (0.42, 0.32);

\coordinate (yx) at ($(x)!0.5!(zx)$);

\pgfmathsetmacro{\halflen}{sqrt((0.11-0.40)^2 + (0.15-0.30)^2)}
\draw[blue!70, thick] (yx) circle (\halflen);

\node[blue!70, font=\normalsize] at ($(yx)+(0.03,-0.15)$) {$B$};

\fill (x) circle (1.3pt);
\node[anchor=north east, font=\normalsize] at ($(x)+(-0.04,-0.02)$) {$x$};

\fill (zx) circle (1.3pt);
\node[anchor=south west, font=\normalsize] at ($(zx)+(0.02,0.02)$) {$z_x$};

\fill (yx) circle (1.3pt);
\node[anchor=south east, font=\normalsize] at ($(yx)+(0,-0.04)$) {$y_x$};

\end{tikzpicture}
    \caption{The construction of the inner balls $B$. Note that the outer ball $\widetilde{B}$ would extend far outside of $B_1$.}
    \label{fig:covering}
\end{figure}
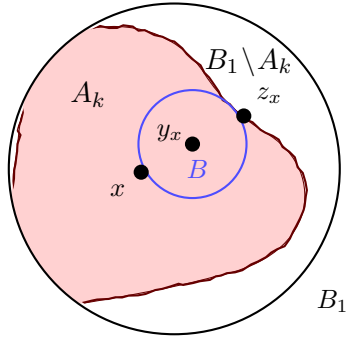

Before continuing with the proof, for the reader's convenience we now recall the definitions of the exponents $\alpha$ and $\kappa$ and define two new exponents $s$ and $\sigma$, as follows:
\begin{align}
    \kappa &\coloneqq \frac{p}{p-(d-1)}, \label{eq: kappadef}
\\
    \alpha & \coloneqq \frac{p}{d-1}, \label{eq: alphadef}\\
    s &\coloneqq \frac{\alpha}{1-\alpha} = \frac{d-1}{p-(d-1)}, \label{eq: sdef}\\
    \sigma &\coloneqq \frac{s}{d-1} = \frac{\kappa}{p} \label{eq: sigmadef}.
\end{align}
At this point it is also convenient to define
$$a_k \coloneqq |A_k|.$$

Now, fix any ball $\widetilde{B} \in \mathcal{G}_k$, and consider its associated inner ball $B$. Owing to the lower-semicontinuity of $u$ and construction of the inner balls $B$, we find that $\ds \inf_{B} u \le t_k$ and, since $2B \subset \widetilde{B}$ is of comparable volume, $\Gamma(2B)\lesssim N_k$. Applying Corollary \ref{cor: sameball}, we obtain
\begin{align} 
\label{eq: firstmeasureestimatelogLeps}
|\{u \le t_k \exp(c_1N_k^\kappa))\} \cap B| & \ge |\{u \le t_k \exp(c_1 \Gamma(2B))\}\cap B|\\
&\ge c_2 \Gamma(2B)^{(1-d)/(1-\alpha)}|B| \ge c_2 N_k^{(1-d)/(1-\alpha)}|B|,
\end{align}
where we have allowed the absolute constants $c_1$ and $c_2$ to pick up some factors only depending on $d$ and $p$, and will continue to do so throughout the proof. Then, using that
$$
N_k \simeq \Gamma(B_5) a_k^{-1/p},
$$
the inequality \eqref{eq: firstmeasureestimatelogLeps} rewrites, after a little bit of exponent algebra, as
$$
|\{u \le t_k \exp(\tilde{c}_1 a_k^{-\sigma})\}\cap B|
\ge
\tilde{c}_2 a_k^s |B|,
$$
where
$$
\tilde{c}_1 \coloneqq c_1 \Gamma(B_5)^\kappa
\qquad\text{and}\qquad
\tilde{c}_2 \coloneqq c_2 \Gamma(B_5)^{(1-d)/(1-\alpha)},$$
and we recall the definitions of the exponents in \eqref{eq: kappadef}, \eqref{eq: alphadef}, \eqref{eq: sdef}, and \eqref{eq: sigmadef}.

This leads us to choose
$$
t_{k+1} \coloneqq t_k \exp(\tilde{c}_1 a_k^{-\sigma}).
$$
Recalling the construction of the balls $B$, we estimate
\begin{align}
a_{k+1}
&= |A_{k+1}| = |\{u> t_{k+1}\}\cap B_1| \\
&\le |\{u>t_k\} \cap B_1|
-
\sum_{\widetilde B \in \mathcal{C}_{i_*}\cap \mathcal{G}_k}
|\{u \le t_{k+1} \} \cap B| \\
&\le |A_k|
-
c_2N_k^{(1-d)/(1-\alpha)}
\left|\bigcup_{\widetilde B \in \mathcal{C}_{i_*}\cap \mathcal{G}_k} B\right| \\
&\le a_k(1-\tilde{c}_2a_k^s),
\end{align}
where we have used \eqref{eq: innerselection}, which implies that an absolute portion of $A_k$ comes from disjoint inner balls associated with good outer balls. We have also used that the balls $B$ corresponding to $\tilde{B} \in \mathcal{C}_{i_*} \cap \mathcal{G}_k$ are disjoint. Of course, we allow the value of the constant $\tilde{c}_2$ to absorb absolute constants.

Now it remains to analyze the coupled system
\begin{align}
\label{eq: adecay}
a_{k+1} &\le a_k(1-\tilde{c}_2a_k^s)\\
\label{eq: tgrowth}
\log(t_{k+1}/t_k) &= \tilde{c}_1a_k^{-\sigma}.
\end{align}

First, we note that $a_k$ indeed tends to zero, with at least polynomial decay, which may be determined from \eqref{eq: adecay} as follows. Setting $b_k \coloneqq a_k^{-s}$ we may apply Bernoulli's inequality, which says that $(1-x)^{-s} \ge 1+sx$ for $x$ small, to find
    \begin{align}b_{k+1} &\ge (1-\tilde{c}_2a_k^s)^{-s}b_k \ge (1+s\tilde{c}_2a_k^s)b_k \\
    &= b_k + s\tilde{c}_2 \ge b_0 + (k+1)s\tilde{c}_2.\\
    \end{align}
Accordingly, 
 \begin{equation}
 \label{eq: adecay2} a_{k+1} \le \frac{1}{(a_0^{-s} + ks\tilde{c}_2)^{1/s}} \lesssim_{\tilde{c}_2,s} k^{-1/s}.
 \end{equation}
Before proceeding, it is convenient to define $w \coloneqq \log u$ and, naturally, $w_k \coloneqq \log t_k$. Then \eqref{eq: tgrowth} becomes
\begin{equation}
\label{eq: tgrowth2}
\Delta w_k \coloneqq w_{k+1} - w_k = \tilde{c}_1a_k^{-\sigma}.
\end{equation}

We claim now that in fact $w_{k} \lesssim a_k^{-s-\sigma}$. This arises naturally from the coupled system. Indeed, \eqref{eq: adecay} implies that $\Delta a_k \lesssim - a_k^{1+s}$, which, combined with \eqref{eq: tgrowth2}, indicates the heuristic $\ds \frac{\Delta w}{-\Delta a} \lesssim a^{-1-\sigma -s}$. To see this more rigorously, apply mean value theorem again to deduce that
$$a_{k+1}^{-s-\sigma} - a_k^{-s-\sigma} = (s+\sigma) \eta_k^{-s-\sigma -1}(a_k-a_{k+1})$$
for some $\eta_k \in [a_{k+1}, a_k]$. Using that $\eta_k \le a_k$, \eqref{eq: adecay}, and \eqref{eq: tgrowth2}, we get
\begin{align}
a_{k+1}^{-s-\sigma} - a_k^{-s-\sigma } &\ge \tilde{c}_2 (s+\sigma)a_k^{-s-\sigma -1} a_k^{1+s}\\
&= \tilde{c}_2 a_k^{-\sigma} = \tilde{c}_2 \tilde{c}_1^{-1}(w_{k+1}-w_k).
\end{align}
Again $\tilde{c}_2$ has absorbed absolute constants. We remark now that $\tilde{c}_2\tilde{c}_1^{-1} \simeq \Gamma(B_5)^{-\kappa + (1-d)/(1-\alpha)} = \Gamma(B_5)^{-\kappa d}$. 
Telescoping, we deduce that 
$$w_{k+1} \lesssim \Gamma(B_5)^{\kappa d}(a_{k+1}^{-s-\sigma} + w_0 - a_0^{-s-\sigma}) \lesssim \Gamma(B_5)^{\kappa d} a_{k+1}^{-s-\sigma}.$$ This was the claimed inequality. Finally, since $a_k \to 0$, the bound $a_k^{-\sigma -s}$ dominates $a_k^{-\sigma} \simeq \Delta w_k$, and we also have\begin{equation}
\label{eq: finalweq}
w_{k+1} = w_k + \Delta w_k \lesssim \Gamma(B_5)^{\kappa d} a_k^{-s-\sigma}.
\end{equation}
Now, for $\ds \eps < \frac{1}{\sigma +s}$, we may estimate
\begin{align}
    \int_{B_1} \log(u)^\eps  &= \int_{B_1} w^\eps = \sum_{k=0}^{K_0}\int_{ \{w_k \le w < w_{k+1}\}\cap B_1} w^\eps  \\
    &\le \sum_{k=0}^{K_0}w_{k+1}^\eps (a_k - a_{k+1})  \lesssim \Gamma(B_5)^{\kappa d \eps}\sum_{k=0}^{K_0} \int_{a_{k+1}}^{a_k} a^{\eps(-s-\sigma)} \ da \\
    & \le \Gamma(B_5)^{\kappa d \eps} \int_0^{a_0} a^{\eps(-s-\sigma)} \ da \lesssim \Gamma(B_5)^{\kappa d \eps}.
\end{align}
where the second line follows from \eqref{eq: finalweq} and the observation that $a^{\eps(-s-\sigma)} \ge a_k^{\eps(-s-\sigma)}$ on the interval $[a_{k+1}, a_k]$. This concludes the proof. 
\end{proof}
As a consequence, we derive a unit-scale improvement of oscillation for solutions to \eqref{eq: degen}, under the assumption that $p>d-1$. 
\begin{proof}[Proof of Proposition \ref{prop: IOL}]
Let $u$ be as in the statement of Proposition \ref{prop: IOL}. Consider the function $v$ by 
$$v \coloneqq \frac{u - \inf_{B_5} u}{\osc_{B_5}u}.$$
Certainly $v$ satisfies $\mathcal{P}^-_{\lambda(\cdot),1}(D^2v)\le 0 \le\mathcal{P}^+_{\lambda(\cdot),1}(D^2v)\le 0$ inside $B_5$. Furthermore, one sees that $\ds \osc_{B_5} v =1$ and $0\le v \le 1$. Now, without loss of generality, we may assume that 
$$|\{v\ge \ds \frac12\} \cap B_1| \ge \frac12 |B_1|$$
for otherwise we could apply the same argument to $w \coloneqq \ds \frac{\sup_{B_5} u - u}{\osc_{B_5} u}.$ Defining $m\ds \coloneqq \inf_{B_1}v$ and noticing that the maximum principle implies that either $m>0$ unless $u$ is constant, we apply Proposition \ref{prop: logL^eps} to $\ds \frac{v}{m}$ to deduce
$$\int_{B_1} \log\left(\frac{v}{m}\right)^\eps \le C(p,d,\eps) \Gamma(B_5)^{\kappa d \eps}.$$
Under the assumption on $v$, we find
$$\frac{|B_1|}{2} \log\left(\frac{1}{2m}\right) ^\eps\le C(p,d,\eps) \Gamma(B_5)^{\kappa d \eps}.$$

As a consequence,
$$m \ge \frac12 \exp\left(-C(d,p,\eps)\Gamma(B_5)^{\kappa d}\right) \coloneqq \theta.$$
Since $0\le v \le 1$, we deduce that $\osc_{B_1}v \le 1-\theta$, and hence
$$\osc_{B_1}u = \sup_{B_1} u - \inf_{B_1} u \le \sup_{B_1} u - \inf_{B_5} u \le (1-\theta) \osc_{B_5}u.$$
\end{proof}
With the above in hand, we include a proof of the Liouville theorem.
\begin{proof}[Proof of Theorem \ref{thm: Liouville}]
After subtracting off a constant, we may assume $\ds \inf_{\R^d} u =0$. For $R>0$ large, define $\theta_R \coloneqq \exp(-c(d,p)\Gamma(B_{5R})^{\kappa d})$, where $c(d,p)$ is as in the above proposition. Applying said proposition to the rescalings $u_R(\cdot) \coloneqq u(R  \cdot)$, we find
$$\osc_{B_R} u \le (1-\theta_R) \osc_{B_{5R}}u.$$
As $R \to \infty$, both $\ds \osc_{B_R}u\to \sup_{\R^d} u$ and $\ds \osc_{B_{5R}}u \to \sup_{\R^d} u$. Meanwhile, the assumption of the theorem guarantees that $\ds \liminf_{R\to \infty} \theta_R \ge \theta_\infty>0$ for some positive number $\theta_\infty$. We thus arrive at 
$$\sup_{\R^d} u \le (1-\theta_\infty) \sup_{\R^d}u$$
which of course implies $u\equiv 0$. 
\end{proof}
\section{A Logarithmic Local Maximum Principle}
\label{sec: LMP}
In this section, we present a new local maximum principle, which holds under the assumption that $\lambda$ satisfies \eqref{eq: keyassumption} for $p>d-1$. The main lemma involves touching a subsolution from above by double exponential blow-up barriers, which is seemingly optimal for geometric reasons that we will discuss after the proof. 

\begin{lem}
\label{lem: LMP1}
Let $u:B_2 \to \R$ be a nonnegative viscosity subsolution to 
$$\mathcal{P}^+_{\lambda(\cdot),1}(D^2u)  \ge 0  \text{ in } B_2$$
such that $u(0)=2$. Assume that the ellipticity $\lambda(\cdot)$ satisfies \eqref{eq: structuralassumption} and \eqref{eq: keyassumption} for some $p>d-1$. Then for any $q>d-1$, there holds
$$1 \le C(d,p,q) \Gamma(B_2)^{d-1} \left(\int_{ \{u>1\}\cap B_2} (\log u)^{q\kappa}\right)^{1/\kappa},$$
where we have defined $\kappa$ as in Proposition \ref{prop: logL^eps}.
\end{lem}
\begin{proof}
As per usual, we may replace $u$ by its supremal convolutions to assume without loss of generality that $u$ is semi-convex and thus twice differentiable except on a set $E$ of measure zero. Once the estimate is proven in this case, the general case will follow after noticing that the sup-convolutions $u^\eps$ given by 
$$u^\eps(x)\coloneqq \max_{y \in B_2} u(y) -\frac{1}{2\eps}|x-y|^2$$
will be locally uniformly bounded and converging pointwise to $u$, which is enough for a simple dominated convergence argument. 

We touch $u$ from above by translations of radial, convex barrier that blows up on $\partial B_1$. For reasons that are postponed until after the proof, we use as our barrier the function $U(x) \coloneqq c\exp(\exp((1-|x|^2)^{-1}))$, where $c = 2\exp(-\exp(16/15))$. For each $x \in B_{1/4}$, consider the sup-convolution
$$h(x) \coloneqq  \max_{z\in B_1} u(z) - U(x-z).$$
Plugging in $z =0$, we find that $h(x) \ge u(0) - U(x) =2 - c\exp(\exp(16/15)) \ge 0$
for all $x \in B_{1/4}$. Clearly the contact points $y$ belong to the interior of $B_{5/4}$, owing to the fact that $U$ tends to $+\infty$ on $\partial B_1$. Finally, from $h\ge 0$ we deduce that $u(y) \ge U(x-y)$ at any contact point $y$. In particular, $u  >1$ on the contact set. 
As in Lemmas \ref{lem: weakharnack1} and \ref{lem: doubling}, we define $\mathcal{T}$ to be the set of contact points $y$ corresponding to every center $x \in B_{1/4}$. For every $y \in \mathcal{T}\setminus E$, we have
\begin{equation}
    \label{eq: finalFOC1}
\nabla u(y) = \nabla U(y-x)
\end{equation}
and
\begin{equation}
    \label{eq: finalFOC2}
    D^2u(y) \le D^2U(y-x).
\end{equation}

The relation \eqref{eq: finalFOC1} and the strict convexity of $U$ imply that for each contact point $y \in \mathcal{T}$, there exists a unique corresponding center $x \in B_{1/4}$. Calling $m(y)=x$ we deduce that for every $y \in \mathcal{T}\setminus E$,
$$D^2u(y) = D^2U(y-m(y))(I-Dm(y))$$
and hence, suppressing the arguments,
$$Dm(y) = I-(D^2U)^{-1}D^2u.$$

In the same way as in the proofs of Lemmas \ref{lem: weakharnack1} and \ref{lem: doubling}, we use the fact that $u$ is a subsolution to deduce that $D^2u(y)$ has at least one non-negative eigenvalue. Then, again as in the proof of Lemma \ref{lem: doubling}, we may argue with \eqref{eq: finalFOC2} and the derived expression for $Dm(y)$ to deduce that $Dm(y)$ has at least one eigenvalue belonging to the interval $[0,1]$.

Using that
$$0 \le \mathcal{P}^+_{\lambda(\cdot),1}(D^2u)\le \tr((D^2u)^+) - \lambda \tr((D^2u)^-)$$
in conjunction with \eqref{eq: finalFOC2} we obtain
$$|D^2u(y)| \le \frac{2}{\lambda(y)}|D^2U(y-m(y))| \text{ for all } y \in \mathcal{T} \setminus E$$
Consequently, we derive the estimate for $y \in \mathcal{T}\setminus E$
\begin{equation}
    \label{eq: finalDmest}
    |Dm(y)| \lesssim \frac{1}{\lambda(y)}\frac{\sigma_{\text{max}}(D^2U)}{\sigma_{\text{min}}(D^2U)}
\end{equation}
where $\sigma_{\text{max}}$ and $\sigma_{\text{min}}$ denote the largest and smallest eigenvalues of $D^2U(y-m(y))$, respectively.

Setting $\vphi(x) = \vphi(|x|) \coloneqq \exp\exp((1-r^2)^{-1})$, it suffices estimate the ratio of the maximal and minimal eigenvalues of $D^2\vphi$. One finds that $D^2\vphi$ has one eigenvalue in the radial direction given by 
$$\vphi''(r) = \vphi(r)\left(\frac{2}{(1-r)^3}\log \vphi(r) + \frac{1}{(1-r)^4} \log \vphi(r) + \frac{1}{(1-r)^4} (\log \vphi(r))^2\right)$$
and $d-1$ tangential eigenvalues given by 
$$\frac{\vphi'(r)}{r} = \frac{\vphi(r)}{r(1-r)^2} \log\vphi(r).$$
Here we have denoted $\exp((1-r)^{-1})$ by $\log \vphi(r)$ to simplify the notation, but it is worth mentioning that $(1-r)^{-1} = \log\log\vphi(r)$. Carrying on, as $r \to 1$, it is clear that the ratio in question is asymptotically equivalent to 
$$\frac{1}{(1-r)^2} \log \vphi(r) = (\log \log \vphi(r))^2 \log \vphi(r).$$

Consequently, \eqref{eq: finalDmest} turns into the estimate
$$|Dm(y)| \lesssim \frac{1}{\lambda(y)} (\log\log U(y-m(y))^2 \log U(y-m(y)) \lesssim_\delta\frac{1}{\lambda(y)} (\log u(y))^{1+\delta}$$
for any $\delta>0$ and every $y \in \mathcal{T}\setminus E$. Here we have used that $u(y) \ge U(y-m(y))$. Arguing then as in Lemma \ref{lem: doubling}, we deduce the estimate $$|\det Dm(y)| \lesssim_\delta \lambda(y)^{1-d}(\log u(y))^{(d-1)(1+\delta)} \text{ for all } y \in \mathcal{T}\setminus E.$$ 

At this point, we would like to argue that $m$ is locally Lipschitz on $\mathcal{T}$, so that we may apply the change of variables formula. In contrast with the situation of Lemma \ref{lem: doubling}, where the contact points took place far from the singularity of the barrier, here we may genuinely have contacts taking place close to $\partial B_1(x)$ for any $x \in B_{1/4}$. To get around this, we observe that $m$ is locally Lipschitz on the sets $\mathcal{T}_N \coloneqq \mathcal{T} \cap \{u\le N\}$, and send $N \to \infty$. To sketch the details, we observe that $U(y-m(y))\le u(y) \le N$ on $\mathcal{T}_N$. Owing to the structure of $U$, certainly $D^2U(\cdot - m(\cdot))$ is bounded on $\mathcal{T}_N$. Then \eqref{eq: finalFOC2} combined with the semi-convexity of $u$ imply that $\nabla u$ is $C^{1,1}$ on $\mathcal{T}_N$. Note that we do not claim any kind of uniformity of these regularity estimates. The estimate for $Dm$ on $\mathcal{T}_N$ is then immediate. Then, using the monotone convergence theorem,
\begin{align}  
|B_{1/4}| &= \lim_{N \to \infty}|m(\mathcal{T}_N)| = \lim_{N \to \infty} \int_{\mathcal{T}_N} |\det(Dm)|\\
&\lesssim  \int_{B_{5/4} \cap \{u>1\}} \lambda^{1-d}(\log u)^{(d-1)(1+\delta)},
\end{align}
from which the claim will follow by H{\"o}lder's inequality. 
\end{proof}
\begin{remark}
    It is natural to ask whether the above estimate can be approved. The answer is essentially negative, at least not by touching with a different barrier. That is, for any convex, radially increasing barrier $U = U(|x|)$ tending to $+\infty$ as $|x| \to 1$, the ratio between the radial and tangential eigenvalues of $D^2U(|x|)$ must be growing at least logarithmically with respect to $U(|x|)$. We can see this as follows. Setting $U(|x|) \coloneqq \vphi(r)$, the relevant ratio (as in \eqref{eq: finalDmest}) is comparable as $r \to 1$ to 
    $$K(\vphi) \coloneqq \frac{\vphi''}{\vphi'}.$$

    Suppose that, as $r \to 1$, there were to hold $K(\vphi) \lesssim (\log \vphi)^\alpha$ for some $\alpha \le 1$. We will see that this is incompatible with $\vphi(r) \to + \infty$ as $r\to 1$. That $\vphi$ is increasing in $r$ allows us to consider to define the variable $s \coloneqq  \vphi(r)$. Setting also $f(r) \coloneqq  \vphi'(r)$, one finds that $f'(s) = \ds \frac{df}{ds} = \frac{df}{dr} \frac{dr}{ds} = \frac{\vphi''(r)}{\vphi'(r)} \lesssim (\log s)^\alpha.$

    Integrating, one obtains that $f(s) \lesssim s(\log  s)^\alpha$ as $s \to \infty$, which is valid under the assumption that $\vphi(r) \to \infty$ as $r\to 1$. In other words, we have derived that $\vphi'(r) \lesssim \vphi(r) (\log \vphi(r))^\alpha$ as $r\to 1$. Separating variables, and integrating from $r_0<1$ to $r$ near $1$, we derive
    $$\int_{\vphi(r_0)}^{\vphi(r)} \frac{1}{\vphi (\log \vphi)^{\alpha}} \ d\vphi \lesssim \int_{r_0}^r \ d\rho \lesssim 1.$$

    If $\alpha \le 1$, we reach a contradiction by sending $r \to 1$ and thus $\vphi(r) \to \infty$. 
\end{remark}
Now we turn to proving Proposition \ref{prop: LMP2}.
\begin{proof}[Proof of Proposition \ref{prop: LMP2}.]
Consider a general nonnegative subsolution $u\ge 0$ as in the statement of the proposition. Define $M \coloneqq \ds \sup_{B_1} u$, and let $x_0 \in \partial B_1$ be a point attaining the supremum. We consider the subsolution $v: B_2 \to \R$
$$v(y) \coloneqq \frac{2u(x_0 + y)}{M},$$
which clearly satisfies the assumptions of Lemma \ref{lem: LMP1}. We let $q>d-1$ and $\kappa = \ds \frac{p}{p-(d-1)}$ be as in the previous lemma. Applying said lemma and estimating naively, one obtains
\begin{equation}
\label{eq: firstlogmomentintegral}
    1 \lesssim \Gamma(B_3)^{d-1} \left(\int_{\{v>1\} \cap B_2} (\log v)^{q\kappa}\right)^{1/\kappa} = \Gamma(B_3)^{d-1}\left(\int_{\{u>M/2\} \cap B_3} \left(\log \left(\frac{2u}{M}\right)\right)^{q\kappa}\right)^{1/\kappa}.
\end{equation}
This proves the first claim of Proposition \ref{prop: LMP2}. It is a homogeneous form of the local maximum principle. 

For the proof of the rest of the proposition, assume that $M>2$. Now, take any exponent $\tau>q\kappa$, and for ease of notation define the moment
$$I_\tau \coloneqq \int_{\{u>1\} \cap B_3} (\log u)^\tau.$$
Integrating layer-by-layer, we find that
\begin{align}
\Gamma(B_3)^{\kappa(1-d)} &\lesssim \int_{\{u>M/2\} \cap B_3} \left(\log\left(\frac{2u}{M}\right)\right)^{q\kappa}\\
&= \int_0^\infty q\kappa t^{q\kappa -1} \left|x \in B_3 : u(x) > \frac{M}{2}e^t\right| \ dt \lesssim  I_\tau \int_0^\infty \frac{t^{q\kappa -1}}{\left(\log(M/2) + t\right)^\tau} \ dt\\
&= c(q,\kappa, \tau) (\log(M/2))^{q\kappa - \tau} I_\tau,
\end{align}
where we have used Chebyshev in the second line and calculated the last integral plainly. The final claim of Proposition \ref{prop: LMP2} follows by rearranging.

\end{proof}
\section{Harnack Inequality}
\label{sec: Harnack}
We now are in the position to prove a Harnack inequality. It will follow by combining the $\log-L^\eps$ estimate of Proposition \ref{prop: logL^eps} with the local maximum principle estimate of Proposition \ref{prop: LMP2}. This requires us to take $p$ large, so that $\eps$ is large and crosses the threshold dictated above. 
\begin{proof}[Proof of Theorem \ref{thm: Harnack}]
Before beginning, we fix an exponent $\eps>0$ as in the $\log-L^\eps$ estimate, satisfying 
\begin{equation}
    \label{eq: epscondition}
    0<\eps < \frac{p-(d-1)}{d},
\end{equation}
as well as an exponent $q>d-1$ as in the local maximum principle. We recall the notation $\ds \kappa =\ds \frac{p}{p-(d-1)}.$

Now, let $u$ be as in the statement of the theorem, normalized such that $\ds \inf_{B_3} u= 1$. The Weak Harnack inequality of Proposition \ref{prop: logL^eps} then implies that there holds 
$$\int_{B_3} (\log u)^\eps \lesssim \Gamma(B_5)^{\kappa d \eps}.$$
We wish to apply Proposition \ref{prop: LMP2}. If $\sup_{B_1} u \le 2$, we are done. So we may as well assume that $\sup_{B_1}u >2$. In this case, we apply the latter local maximum principle estimate of Proposition \ref{prop: LMP2} with $\tau = \eps$, which is legitimate provided that $\eps>q\kappa$, to deduce that
$$\sup_{B_1} u \le c_1 \exp\left(c_2 \Gamma(B_5)^{\frac{\kappa(d-1)}{\eps - q\kappa}+ \kappa \eps d}\right).$$

Requiring that $q>d-1$ and that $\eps$ satisfy \eqref{eq: epscondition} leads to a restriction on the set of admissible $p$ such that $\eps > q \kappa$, namely, that
    $$(p-(d-1))^2 > pd(d-1).$$
    One can check via the quadratic formula that this is precisely the condition 
    $$p > \frac{d-1}{2}\left((d+2) + \sqrt{d(d+4)}\right),$$
as in the statement of the theorem. This concludes the proof of the Harnack inequality.
\end{proof}
\section{Other Notions of Solution}
\label{sec: othernotions}
The standing assumption of this paper was that the ellipticity $\lambda(\cdot)$ is continuous, strictly positive on the interior, and satisfies an integrability condition of the form $\lambda^{-1} \in L^p(B_5).$
Under this assumption, the results were proven for so-called $C$-viscosity solutions. We now discuss the possibility of removing the continuity and strict positivity assumption on $\lambda(\cdot)$, for two different notions of solution. As we saw in the introduction, this is not possible for $C$-viscosity solutions. For the remainder of this section, we consider ellipticities $\lambda(\cdot)$ satisfying only the integrability condition
$$\frac1\lambda \in L^p(B_5).$$

The simplest model degenerate elliptic equation with ellipticity that may actually vanish at a point is given by 
\begin{equation}
    \label{eq: toymodel}
    |x|^\alpha u''(x) =0 \text{ on } (-1,1),
\end{equation}
as written in the introduction. There, we saw that $u(x) = |x|$ is a classical $C-$viscosity solution which will violate any possible Harnack inequality for all $\alpha>0$. If we wish to develop a Harnack inequality for potentially vanishing ellipticities, we need our class of solutions to exclude this counterexample. 

The first new notion of solution we explore is called a \textit{$L^s$-strong solution}. We say that $u \in W^{2,s}(B_5)$ is an $L^s$-strong supersolution to
$$\mathcal{P}^-_{\lambda(\cdot),1}(D^2u) \le 0 \text{ in } B_5$$
if it satisfies
$$\mathcal{P}^-_{\lambda(x),1}(D^2u(x))= \lambda(x) \tr((D^2u)^+(x)) - \tr((D^2u^-)(x))\le 0 \text{ a.e.  in }B_5,$$
with an analogous definition for subsolutions. This is a \textit{stronger} notion of solution, in that $L^s$-strong solutions are necessarily $C-$viscosity solutions, and come equipped with more regularity, provided $s>d/2$. The interested reader should consult \cite{GilbargTrudinger2001} for more on strong solutions.

We claim that all of our results hold for $L^s$-strong solutions, if $s>d$. Notice that the function $u(x) =|x|$ does not belong to $W^{2,s}((-1,1))$ for any $s\ge 1$, as its second derivative carries a Dirac mass at the origin. 

\begin{thm}
    Fix $s>d$. Then all of the results of this paper remain true for $L^s$-strong solutions $u$, without requiring that the ellipticity $\lambda(\cdot)$ fulfill the assumption \eqref{eq: structuralassumption}. In particular, our results hold for ellipticities which may vanish on sets of measure zero. 
\end{thm}
\begin{proof}
    The proofs are nearly identical. Instead of reducing to semi-concave and semi-convex approximations, one may perform identical calculations directly for the solution $u$, since the set where $\lambda(\cdot)$ is positive and $u$ is second-differentiable is full measure within $B_5$, and its Lebesgue-null complement will not affect the measure estimates used to obtain the weak Harnack inequality or local maximum principle. The condition $s>d$ arises when one applies the change of variables formula. It is not hard to see that in Lemmas \ref{lem: weakharnack1}, \ref{lem: doubling}, and \ref{lem: LMP1}, the change of variables all are $W^{1,s}$ transformations. 
\end{proof}

There is yet another notion of viscosity solution that we here explore, namely the notion of \textit{$L^s$-viscosity solution} introduced by Caffarelli, Crandall, Kocan, and {\'S}wiech in the paper \cite{CaffarelliCrandallKocanSwiech1996}, which we recall now. We say that a continuous function $u$ is an $L^s$-viscosity supersolution to 
$$\mathcal{P}_{\lambda(\cdot),1}^-(D^2u) \le 0 \text{ in } B_5$$
if for all $\vphi \in W^{2,s}_{\text{loc}}(B_5)$ and for all $x \in B_5$ such that $u-\vphi$ has a local minimum at $x$, there holds 
\begin{equation}
\label{eq: L^sviscositysuper}
\text{ess}-\liminf_{y\to x} \mathcal{P}_{\lambda(y),1}^-(D^2\vphi(y)) \le 0
\end{equation}
Likewise, we say that $u$ is an $L^s$-viscosity subsolution to 
$$\mathcal{P}_{\lambda(\cdot),1}^+(D^2u) \ge 0 \text{ in } B_5$$
if for all $\vphi \in W^{2,s}_{\text{loc}}(B_5)$ and for all $x \in B_5$ such that $u-\vphi$ has a local maximum at $x$, there holds 
\begin{equation}
    \label{eq: L^sviscositysub}
    \text{ess}-\limsup_{y\to x} \mathcal{P}_{\lambda(y),1}^+(D^2\vphi(y)) \ge 0.
\end{equation}
Of course, one needs $s>d/2$ to make sense of touching by a $W^{2,s}$ function. This is a weaker notion of solution than the aforementioned $W^{2,s}$-strong solution, but it is a stronger notion of solution than the classical $C$-viscosity solution. Now we briefly explore whether the results of this paper hold for $L^s$-viscosity solutions, without continuity or strict positivity assumptions on $\lambda(\cdot)$. 
\\\par Consider again the toy model \eqref{eq: toymodel} and the function $u(x)=|x|$. Since $W^{2,s}((-1,1))$ functions are locally $C^1$, there are no admissible test functions $\vphi$ touching $u$ from above at the origin, and thus the subsolution condition is vacuously satisfied. We claim now that if $\alpha s \ge 1$, then $u$ is also an $L^s$-viscosity supersolution, and thus a full $L^s$-viscosity solution. To that end, suppose that $\vphi$ touches $u$ from below at the origin, but that for some $\delta>0$,
$$|x|^\alpha \vphi''(x) >\delta$$
for almost every $x$ in some neighborhood of the origin $(-r_0, r_0)$. This plainly precludes $\vphi''$ from belonging to $L^s((-1,1))$, as $|\vphi''(x)|^s \gtrsim |x|^{-\alpha s}$ near the origin, and the latter is not integrable in the regime $\alpha s \ge 1.$ We conclude that $u(x)=|x|$ must be an $L^s$-viscosity solution, and therefore the assumption that $\lambda(\cdot)$ is a positive continuous function may not be dropped for this $L^s$-viscosity solutions if $s$ is too large. 

There is a moral reason for why Harnack inequality and our techniques may fail for $C$-viscosity solutions or $L^s$-viscosity solutions if the ellipticity is allowed to vanish at points, which is that infimal convolutions of supersolutions may fail to satisfy the appropriate supersolution condition. A simple calculation for $u(x) =|x|$ shows that
$$u_\eps(x) \coloneqq \inf_{y \in \R} |y| + \frac{1}{2\eps}|x-y|^2 = \begin{cases} \frac{|x|^2}{2\eps} & \text{ for } |x| \le \eps \\ |x| - \frac\eps2 & \text{ for } |x|>\eps.\end{cases},$$
which evidently does not satisfy the $L^s-$viscosity supersolution condition \eqref{eq: L^sviscositysuper}. This is due to the fact that all of the parabolas centered at $|x|<\eps$ touch at the origin, where the coefficient vanishes. 
\\\par On the other hand, something interesting occurs in the case $\alpha s < 1$, where the class of admissible $W^{2,s}$ touching functions becomes rich enough to impose further structural conditions on the solution $u$. It is clear that a solution $u$ to \eqref{eq: toymodel} must be of the form
$$u(x) = \begin{cases}  u_L x & \text{ for } x\le 0 \\  u_R x & \text{ for } x>0\end{cases}$$
after subtracting off a constant. If $\alpha s<1$, then cusps of the form $x \mapsto |x|^{2-\alpha}$ become admissible test functions. We claim that the supersolution condition \eqref{eq: L^sviscositysuper} implies that $u_L \ge u_R$, and the subsolution condition \eqref{eq: L^sviscositysub} implies that $u_R \ge u_L$. Together, they will imply that $u$ is linear. We'll just analyze the former case, as the latter follows by symmetrical arguments. To that end, suppose for the sake of contradiction that $u_L <u_R$, and select $b \in (u_L, u_R)$. We use as a test function $\vphi(x)\coloneqq bx + c|x|^{2-\alpha}$ for $c$ small, which touches $u$ from below at the origin. But

$$\liminf_{x \to 0} |x|^\alpha \vphi''(x) = c(2-\alpha)(1-\alpha) >0,$$
contradicting that $u$ satisfies \eqref{eq: L^sviscositysuper}. Applying a similar argument from above, we conclude that $u_R=u_L$. 

The above calculation shows that if the results of this paper are to be extended to $L^s$-viscosity solutions to degenerate equations with possibly vanishing ellipticities $\lambda(\cdot)$, the exponent $s$ must be small enough relative to the integrability of $\lambda^{-1}$. It also suggests that one must exploit the admissible class of ``singular" test functions in a way to deduce further structural conditions on the solution $u$, but it is not clear at all what this would look like in general. Notice that all of our arguments in this paper involve testing with smooth functions, or at least functions which are smooth in neighborhoods of contact. It is also unclear whether one may use inf-convolution approximation arguments for $L^s$-viscosity solutions in any case. We leave this problem for future work.
\bibliographystyle{abbrv}
\bibliography{cite}

\end{document}